\documentclass[11pt]{amsart} 
\usepackage{amsthm,amsbsy,amsfonts,amssymb,amscd}

\usepackage[mathscr]{euscript} 

\usepackage{enumerate} 
\usepackage{tikz} 
\usetikzlibrary{calc,intersections,through, backgrounds,arrows,decorations.pathmorphing,backgrounds,positioning,fit,petri} 
\usetikzlibrary{calc}

\usepackage{tikz-cd} 
\usepackage[all]{xy}
\xyoption{knot}
\xyoption{poly}

\usepackage{hyperref}

\usepackage[margin=1in]{geometry}

\newcommand{\cIG}{\mathcal{IG}}
\newcommand{\cGI}{\mathcal{GI}}

\newcommand{\Ab}{\mathbf{Ab}}
\newcommand{\tp}{{\rm tp}}

\newcommand{\pd}{{\rm pd}} 

\newcommand{\HH}{\mathsf{H}}

\newcommand{\EE}{{\mathsf{E}}}
\newcommand{\FF}{{\mathsf{F}}}
\newcommand{\AAf}{{\mathsf{A}}}
\newcommand{\BB}{{\mathsf{B}}}
\newcommand{\KK}{{\mathsf{K}}}
\newcommand{\TT}{{\mathsf{T}}}

\newcommand{\RR}{\mathsf{R}}
\newcommand{\II}{\mathsf{I}}
\newcommand{\CC}{\mathsf{C}}
\newcommand{\NN}{\mathsf{N}}

\newcommand{\sF}{\mathsf{F}}

\newcommand{\gd}{\mathrm{glo.d}}

\newcommand{\fd}{\mathrm{fin.d}}
\newcommand{\dd}{\mathbf{d}}

\newcommand{\cmE}{\mathscr{E}}
\newcommand{\cmS}{\mathscr{S}}

\newcommand{\lfp}{$\mathsf{lfp}\;$}
\newcommand{\lfpp}{$\mathsf{lfp}$}

\newcommand{\Gd}{{\rm{Gd}}}

\newcommand{\Res}{{\rm Res}^{\fg}_{\fg_{\oa}} }

\newcommand{\Ind}{{\rm Ind}^{\fg}_{\fg_{\oa}} }
\newcommand{\Coind}{{\rm Coind}^{\fg}_{\fg_{\oa}} }
\newcommand{\coind}{{\rm Coind} }

\newcommand{\oa}{\bar{0}}
\newcommand{\ob}{\bar{1}}

\newcommand{\tto}{\twoheadrightarrow}

\newcommand{\cB}{\mathcal B}
\newcommand{\cO}{\mathcal O}
\newcommand{\scO}{s\hspace{-0.1mm}\mathcal{O}}

\newcommand{\cP}{\mathcal{P}}
\newcommand{\cGP}{\mathcal{G\hspace{-.2mm}P}}

\newcommand{\im}{{\rm{im}\,}}
\newcommand{\id}{\mathrm{id}}
\newcommand{\Id}{\mathsf{Id}}

\newcommand{\ad}{\mathrm{ad}}

\newcommand{\Hom}{\mathrm{Hom}} 
\newcommand{\Ext}{\mathrm{Ext}} 
\newcommand{\GE}{\mathrm{G\hspace{-0.4mm}Ext}}

\newcommand{\End}{\mathrm{End}} 
\newcommand{\coker}{\mathrm{coker}\,}

\newcommand{\fk}{\mathfrak{k}}
\newcommand{\dfs}{{/\kern-2pt/}}

\newcommand{\op}{{\rm op}}
\newcommand{\mk}{\Bbbk}
\newcommand{\cC}{\mathcal{C}}
\newcommand{\cL}{\mathcal{L}}
\newcommand{\cD}{\mathcal{D}}

\newcommand{\cA}{\mathcal{A}}

\newcommand{\add}{{\rm{add}}}

\newcommand{\fn}{\mathfrak{n}}
\newcommand{\fg}{\mathfrak{g}}
\newcommand{\fu}{\mathfrak{u}}
\newcommand{\fl}{\mathfrak{l}}
\newcommand{\fp}{\mathfrak{p}}
\newcommand{\fh}{\mathfrak{h}}
\newcommand{\fb}{\mathfrak{b}}
\newcommand{\mN}{\mathbb{N}}
\newcommand{\mC}{\mathbb{C}}
\newcommand{\mZ}{\mathbb{Z}}

\numberwithin{equation}{section}

\newtheoremstyle{notes} {} {} {} {} {\bfseries} {.} {.5em} {}

\theoremstyle{plain}

\newtheorem{prop}[subsubsection]{Proposition}
\newtheorem{scholium}[subsubsection]{Scholium}

\newtheorem{lemma}[subsubsection]{Lemma}

\newcommand{\tr}{{\rm{tr}}}

\newtheorem{cor}[subsubsection]{Corollary}

\newtheorem{thm}[subsubsection]{Theorem}

\theoremstyle{remark}

\newtheorem{rem}[subsubsection]{Remark} 

\newtheorem{ddef}[subsubsection]{Definition} 

\swapnumbers

\usepackage{etoolbox}

\usepackage{ytableau}
\usepackage{youngtab}

\makeatletter
\pretocmd{\appendix}{\addtocontents{toc}{\protect\addvspace{10\p@}}}{}{}
\makeatother

\theoremstyle{remark}

\newtheorem{ex}[subsubsection]{Example}

\newtheoremstyle{construction} {} {} {} {} {\bfseries} { } {0pt} {}

\theoremstyle{construction}

\title[Gorenstein homological algebra]{Gorenstein homological algebra for rngs and Lie superalgebras}

\author{Kevin~Coulembier}

\newcommand{\ind}{{\rm Ind}}
\newcommand{\res}{{\rm Res}}

\keywords{Gorenstein projective objects, Gorenstein extension groups, Gorenstein category, Frobenius extensions, rngs with enough idempotents, Lie superalgebras, Serre functor, category~$\cO$, Tate cohomlogy}

\subjclass[2010]{17B55; 18G15; 18G25; 17B10}

\begin{document} 
\date{} 
\begin{abstract}
We generalise notions of Gorenstein homological algebra for rings to the context of arbitrary abelian categories. The results are strongest for module categories of rngs with enough idempotents.
We also reformulate the notion of Frobenius extensions of noetherian rings into a setting which allows for direct generalisation to arbitrary abelian categories.

The abstract theory is then applied to the BGG category~$\cO$ for Lie superalgebras, which can now be seen as a ``Frobenius extension" of the corresponding category for the underlying Lie algebra and is therefore ``Gorenstein". In particular we obtain new and more general formulae for the Serre functors and instigate the theory of Gorenstein extension groups. 
	\end{abstract}

\maketitle 


\section*{Introduction} 

The original motivation for this paper is the following {\bf observation}: ``For an arbitrary module~$M$ in the BGG category~$\cO$ of a classical Lie superalgebra, see~\cite{BGG, preprint, book}, the projective dimension of the restriction of $M$ to the underlying Lie algebra (inside category~$\cO$ for this Lie algebra) is an intrinsic categorical property of $M$ itself''. Note that the construction of such categorical invariants is motivated by the ongoing work on equivalences between blocks in category~$\cO$, see e.g.~\cite{BG, CMW, CS}. Although this observation is not particularly difficult to prove directly, it is rather surprising since restriction to the underlying Lie algebra is a purely Lie algebraic manipulation, with {\it a priori} no intrinsic categorical interpretation.

In Part~I of the paper, we develop ideas from Frobenius extensions and Gorenstein homological algebra in ring theory to arbitrary abelian categories. 
In Part~II, we apply this to module categories over Lie superalgebras, which explains in particular the above observation in a general framework.

In Section~\ref{SecFE}, we propose a notion of a ``Frobenius extension of an abelian category''. When applied to the category of finitely generated modules over noetherian rings, we find that the Frobenius extensions of $S$-mod, are precisely the categories ${R}$-mod for all Frobenius extensions $R$ of $S$ in the classical sense. This demonstrates the consistency of the new definition with the classical one.

In Section~\ref{SecGHA}, we generalise some results on Gorenstein homological algebra for rings from \cite{AB, Avramov, Holm, Iwanaga} to arbitrary abelian categories. The generalisation are immediate, without serious modification of proofs. 
We define Gorenstein projective objects, G-dimensions, and Gorenstein and Tate extension groups. The latter two form a long exact sequence with the ordinary extension groups, generalising one of the main results of \cite{Avramov}. As generalisations of the notions for rings, see e.g.~\cite{AR, Avramov}, we define notions of ``Gorenstein'' and ``Iwanaga-Gorenstein'' categories.

In Section~\ref{SecIdem}, we focus on module categories over rngs (r$i$ngs without $i$dentity) with enough idempotents which satisfy some noetherian property. As a special case we consider, in Section~\ref{SecLFP}, locally finite $\mk$-linear abelian categories with enough projective and injective objects. The class of such categories is closed under Frobenius extensions. Furthermore, the Frobenius extension is Gorenstein if and only if the original category was. In this setting, the G-dimension of an object in the extension is the G-dimension of the ``image'' of that object in the underlying category.

In Sections~\ref{SecLSA} and \ref{SecO}, we focus on representations of Lie superalgebras, mainly in category~$\cO$. It is known that the universal enveloping algebra of a Lie superalgebra is a Frobenius extension (in the classical sense) of that of the underlying Lie algebra, see~\cite{Bell}. Consequently, category~$\cO$ for the Lie superalgebra is a Frobenius extension, in the new sense, of category~$\cO$ for the Lie algebra. Category~$\cO$ for a Lie algebra is Gorenstein in an obvious degenerate way, such that G-dimensions coincide with projective dimensions. The results in Part I thus imply that category~$\cO$ for a Lie superalgebra is Gorenstein, where the G-dimension of a module is equal to the projective dimension of its restriction as a Lie algebra module. This explains the original observation.
As an extra result, we find a new characterisation of Serre functors, simplifying and extending some results in \cite{Serre}. 

In Section~\ref{SecGL}, we initiate the study of Gorenstein extensions and G-dimensions in category~$\cO$ over~$\mathfrak{gl}(m|n)$.

\part{Abstract Gorenstein homological algebra}

\section{Preliminaries}
\subsection{Categories and functors} We let $\cC$ denote an arbitrary abelian category. 
\subsubsection{} In general we do not assume existence of projective or injective objects, we thus define extension functors as in~\cite[Section~3.4]{Weibel}.
For $X\in\cC$, the projective dimension~$\pd_{\cC}X\in\mN\cup\{\infty\}$ is the minimal $i$ for which $\Ext^j_{\cC}(X,-)=0$ for all $j>i$ and similarly for the injective dimension~$\id_{\cC}X$. Furthermore, the global (resp. finitary) dimension of the category,~$\gd\cC\in\mN\cup\{\infty\}$ (resp. $\fd\cC\in\mN\cup\{\infty\}$), is the supremum of all (resp. all finite) projective dimensions of the objects in~$\cC$. 

\subsubsection{}
For a chain complex $\{d_i:C_i\to C_{i-1}\}$ in $\cC$ we have syzygies $\Omega^i(C_\bullet)=\coker d_{i+1}$ and homologies $\HH_i(C_\bullet)=\ker d_i/\im d_{i+1}$. For a cochain complex $\{\partial^{i}:C^{i-1}\to C^{i}\}$, we have cohomologies $\HH^i(C^\bullet)=\ker\partial^{i+1}/\im\partial^i$ and cocycles $\CC^i(C^\bullet)=\ker \partial^{i+1}$.
 
\subsubsection{} We let $p\cC$, resp. $i\cC$, denote the full subcategory of projective, resp. injective, objects in~$\cC$. We denote the full subcategory of objects with finite projective dimension by $p\cC^{(-)}$. 

Assuming that~$\cC$ has enough projective objects, for any $n\in\mN$, the full subcategory of $\cC$ of all objects of the form $\Omega^n(P_\bullet)$, for~$P_\bullet$ a projective resolution of some $X\in\cC$ is denoted by $\Omega^n\cC$. 

We denote the full subcategory of~$\cD^b(\cC)$ of complexes quasi-isomorphic to perfect complexes (finite complexes of projective objects) by~$\cD_{per}(\cC)\cong K^b(p\cC)$.

\subsubsection{}Functors are assumed to be covariant unless specified otherwise. Whenever we are working in additive (resp. $\mk$-linear) categories, functors are assumed to be additive (resp. $\mk$-linear). For a $\mk$-linear additive category~$\cD$ with finite dimensional homomorphism spaces, a {\bf Serre functor} $\Phi$ on~$\cD$, following~\cite[Section~I.1]{VDB}, is an autoequivalence of~$\cD$ equipped with isomorphisms
$$\xi_{X,Y}: \Hom_{\cD}(X,\Phi Y)\;\;\tilde\to\;\; \Hom_{\cD}(Y,X)^\ast,$$
for all $X,Y\in\cD$, natural in both $X$ and $Y$.
By~\cite[Lemma~I.1.3]{VDB}, a Serre functor on~$\cD$, if it exists, is unique up to isomorphism. 

\subsection{Elementary diagram constructions}
Fix an abelian category~$\cC$.
\begin{lemma}\label{baby}[Baby horseshoe lemma]
Consider short exact sequences $X\hookrightarrow P\tto X_1$,~$Z\hookrightarrow Q\tto Z_1$ and~$X\hookrightarrow Y\tto Z$ in~$\cC$ and assume that~$\Ext^1_{\cC}(Z,P)=0$. Then there exists $Y_1$ such that we have a commuting diagram with exact columns and rows:
{{\footnotesize$$\xymatrix{&0\ar[d]&0\ar[d]&0\ar[d]\\
0\ar[r]&X\ar[r]\ar[d]&P\ar[r]\ar[d]&X_1\ar[r]\ar[d]&0\\
0\ar[r]&Y\ar[r]\ar[d]&P\oplus Q\ar[r]\ar[d]&Y_1\ar[r]\ar[d]&0\\
0\ar[r]&Z\ar[r]\ar[d]&Q\ar[r]\ar[d]&Z_1\ar[r]\ar[d]&0\\
&0&0&0
}$$}}
\end{lemma}
\begin{proof}
Since $\Ext^1_{\cC}(Z,P)=0$, we have $\Hom_{\cC}(Y,P)\tto \Hom_{\cC}(X,P)$. Take $\alpha:Y\to P$ in the preimage of $X\hookrightarrow P$. Compose $Y\tto Z$ and~$Z\hookrightarrow Q$ to construct $\beta:Y\to Q$. Then we have a monomorphism $\alpha+\beta:Y\to P\oplus Q$, admitting two commuting squares in the diagram. The existence of~$Y_1$ and the short exact sequence in the last column then follow from the snake lemma.
\end{proof}

\begin{lemma}\label{baby2}[Baby comparison lemma]
Consider short exact sequences $Y\hookrightarrow P\tto Y_1$,~$Z\hookrightarrow Q\tto Z_1$, a morphism $\alpha:Y\to Z$ in~$\cC$ and assume that~$\Ext^1_{\cC}(Y_1,Q)=0$. Then there exists a commuting diagram with exact rows:
$$\xymatrix{
0\ar[r]&Y\ar[r]\ar[d]^{\alpha}&P\ar[r]\ar[d]&Y_1\ar[r]\ar[d]&0\\
0\ar[r]&Z\ar[r]&Q\ar[r]&Z_1\ar[r]&0.
}$$
Moreover, if $\alpha$ is an epimorphism, the diagram contains an exact sequence
$$0\,\to\, \ker\alpha\,\to\, P\,\to\, Q\oplus Y_1\,\to \,Z_1\,\to\, 0.$$
\end{lemma}
\begin{proof}
By $\Ext^1_{\cC}(Y_1,Q)=0$, we have $\Hom_{\cC}(P,Q)\tto \Hom_{\cC}(Y,Q)$. This allows to construct a morphism $\beta:P\to Q$ to make a commuting square with $\alpha$. 
We have an exact sequence
$$0\to \Hom_{\cC}(Y_1,Z_1)\to\Hom_{\cC}(P,Z_1)\to \Hom_{\cC}(Y,Z_1).$$
Let $\beta'\in\Hom_{\cC}(P,Z_1)$ denote the composition of~$\beta$ with $Q\tto Z_1$. By construction,~$\beta'$ is mapped to zero in the above sequence, meaning we can consider a non-zero $\gamma\in\Hom_{\cC}(Y_1,Z_1)$ in its preimage, which completes the commuting diagram.

Now we construct the exact sequence. Consider a morphism $\phi:P\to Q\oplus Y_1$ obtained by adding the morphisms $P\to Q$ and $P\to Y_1$ in the diagram. It is immediate by construction that~$\ker\phi\cong\ker\alpha$ and $\coker\phi$ is isomorphic to the cokernel of the composition of $\alpha$ with $Z\hookrightarrow Q$. When $\alpha$ is an epimorphism it thus follows that~$\coker \phi\cong Z_1$.
\end{proof}

\subsection{Subcategories of abelian categories}
For the remainder of this section we fix an arbitrary full subcategory~$\cB$ of an abelian category~$\cC$. 

\subsubsection{} Assume $\cC$ has enough projective objects. Following~\cite{AB} or~\cite[Section~3]{AR}, a full karoubian (idempotent split) subcategory~$\cB$ of~$\cC$ is {\bf (projectively) resolving} if it contains all projective objects in $\cC$ and for every short exact sequence
$$0\to X\to Y\to Z\to 0$$
in~$\cC$ with $Z\in \cB$, we have that~$X\in \cB$ if and only if $Y\in\cB$. 

\subsubsection{} We let ${}^{\perp}\cB$, resp. $\cB^{\perp}$, denote the full subcategory of objects $X\in \cC$ which satisfy $\Ext_{\cC}^i(X,Y)=0$, resp. $\Ext_{\cC}^i(Y,X)=0$, for all $Y\in\cB$ and all $i>0$. If we only require first extensions to vanish we write ${}^{\perp_1}\cB$, or $\cB^{\perp_1}$. A chain complex~$C_\bullet$ in $\cC$, is called {\bf left}, resp. {\bf right,~$\cB$-acyclical} if $\Hom_\cC(Y, C_\bullet)$, resp. $\Hom_{\cC}(C_\bullet, Y)$, is exact for all $Y\in\cB$.

\begin{lemma}\label{secPacyc}
An exact sequence $P_\bullet$ of projective objects in~$\cC$ is right $\cB$-acyclical if and only if each syzygy object is in~${}^\perp \cB$.
\end{lemma}
\begin{proof}
Take $i\in\mZ$. The truncated complex~$\cdots\to P_{i+1}\to P_{i}\to 0$ is a projective resolution of $X:=\Omega^i(P_\bullet)$. Hence we find~$\Ext^k_{\cC}(X, Y)\cong \HH^{i+k}(\Hom_{\cC}(P_\bullet,Y))$, for~$k\ge 1$ and~$Y\in \cB$.
\end{proof}

\begin{ddef}\label{properDef}
A {\bf $\cB$-resolution} of an object $X\in \cC$ is a chain complex
$$\cdots\to B_{2}\to^{d_2} B_1\to^{d_1} B_0\to^{d_{0}} 0$$ in~$\cB$, with only non-zero homology given by~$\HH_0(B_\bullet)\cong X\cong\Omega^0(B_\bullet)$.
A $\cB$-resolution~$B_\bullet$ of~$X$ is {\bf proper} if the augmented exact sequence $B_\bullet\to X\to 0$ is left $\cB$-acyclical.
\end{ddef}

\subsubsection{}Following~\cite[Section~1]{AR}, a {\bf right $\cB$-approximation} of~$X\in\cC$ is an epimorphism $\alpha: A\tto X$ with $A\in\cB$, such that the induced morphism $\Hom_{\cC}(A',A)\to \Hom_{\cC}(A',X)$ is surjective for each $A'\in\cB$. The subcategory~$\cB$ is {\bf contravariantly finite} in $\cC$ if each object in~$\cC$ admits a right $\cB$-approximation.
An epimorphism $\alpha: A\tto X$, for~$A\in\cB$ such that~$K:=\coker (\alpha)$ is in~${\cB}^{\perp_1}$ is called a {\bf special} right $\cB$-approximation.

\subsubsection{} A {\bf cotorsion pair} in an abelian category is a pair $(\cA,\cB)$ of subcategories such that~$\cA={}^{\perp_1}\cB$ and~$\cB=\cA^{\perp_1}$. A cotorsion pair is {\bf hereditary} if, in addition, the restriction of  $\Ext^j_{\cC}(-,-)$ to~$\cA^{\op}\times \cB$ vanishes for all $j>0$. A cotorsion pair {\bf admits enough projectives} if every object $X\in\cC$ admits a special right $\cA$-approximation. Here, this is $B\hookrightarrow A\tto X$, with $A\in\cA$ and $B\in\cB$.

\subsection{Noetherian rings} By ring we always mean unital ring. 
The category of all left modules over a ring~$R$ is denoted by~$R$-Mod.
By noetherian ring we mean a ring which is both left and right noetherian. For a noetherian ring $R$, we denote by~$R$-mod the (abelian) category of finitely generated modules. 
A {\bf finite ring extension} of a ring $S$ is a ring $R$ such that~$S$ is a (unital) subring and such that the (left) $S$-module~$R$ is finitely generated.
\begin{ddef}\label{DefFEring}
Consider a noetherian ring $S$ with automorphism $\alpha$. A ring $R$ is an {\bf $\alpha$-Frobenius extension of~$S$} if 
it is a finite ring extension, with $R$ projective as a left $S$-module and for which we have an $(R,S)$-bimodule morphism 
\begin{equation}\label{eqbimod}R\;\cong\; \Hom_S(R,{}_\alpha S).\end{equation}
\end{ddef}

Such an isomorphism~$\varphi: R\to \Hom_S(R,{}_\alpha S)$, $a\mapsto\varphi_a$, for all $a\in R$, leads to a bi-additive map
$$\sigma:R\times R\to S,\qquad \sigma(a,b)=\varphi_b(a),$$
which satisfies $\sigma(ab,c)=\sigma(a,bc)$, for~$a,b,c\in R$, and $\sigma(xa,b)=\alpha(x)\sigma(a,b)$ and $\sigma(a,bx)=\sigma(a,b)x$, for~$x\in S$. By \cite[Theorem~1.1]{Bell},~$\sigma$ is non-degenerate.

\begin{ex}\label{exAlgebra}
Let $\mk$ be any field and $A$ a finite dimensional $\mk$-algebra. By the above, it follows easily that~$A$ is a {\bf Frobenius algebra} if and only if it is an $\id_{\mk}$-Frobenius extension of~$\mk$.
\end{ex}

For (associative) algebras we will in general {\bf not} assume that they are unital or finite dimensional. A $\mk$-algebra is thus a $\mk$-vector space with a bilinear associative product.


\section{Frobenius extensions of abelian categories}\label{SecFE}

\subsection{Definitions}

We let $\cC$ be an arbitrary abelian category. 
\begin{ddef}\label{DefFE}
A {\bf Frobenius extension} of~$\cC$ is an abelian category~$\widetilde\cC$ with (additive) functor
$\RR:\widetilde\cC\to \cC$ such that
\begin{enumerate}[(i)]
\item $\RR$ has a left adjoint $\II$ and a right adjoint $\CC$;
\item $\RR,\II$ and $\CC$ are faithful;
\item $\II\cong\CC\circ \EE$ for an auto-equivalence $\EE$ of~$\cC$.
\end{enumerate}
\end{ddef}
Clearly $\widetilde{\cC}$ is a Frobenius extension of~$\cC$ if and only if $\widetilde{\cC}^{\op}$ is a Frobenius extension of~$\cC^{\op}$.
We will show in Theorem~\ref{ThmEquivFro} that this is an appropriate generalisation Definition~\ref{DefFEring}. In analogy with the situation in ring theory, one could use the term $\EE$-Frobenius extension of a category~$\cC$.
\begin{lemma}\label{LemBasix}
Consider the data of Definition~\ref{DefFE}. 
\begin{enumerate}[(i)]
\item The functors~$\RR,\II,\CC$ are exact.
\item The functors~$\RR,\II,\CC$ restrict to functors between $p\widetilde{\cC}$ and $p\cC$, and between $i\widetilde\cC$ and $i\cC$.
\item The units $\Id_{\cC}\to \RR\circ\II$ and $\Id_{\widetilde\cC}\to \II\circ\CC$ are monomorphisms.
\item The counits $\II\circ\RR\to \Id_{\widetilde\cC}$ and $\RR\circ\CC\to\Id_{\cC}$ are epimorphisms.
\end{enumerate}
\end{lemma}
\begin{proof}
Parts (i) and (ii) are standard homological properties of adjoint functors.
For parts (iii) and (iv) we just prove that~$\eta: \Id_{\cC}\to \RR\circ\II$ is a monomorphism, the other claims can be proved identically. As $(\II,\RR)$ is a pair of adjoint functors, for arbitrary~$X\in\cC$, we have an isomorphism
$$\nu:\Hom_{\widetilde\cC}(\II-,\II X)\;\;\tilde\to\;\;\Hom_{\cC}(-,\RR\circ\II X),$$
of contravariant functors. By definition, we have $\eta_X=\nu_X(1_{\II X})$ and, for an arbitrary morphism $\alpha: Y\to X$ in $\cC$, the diagram
$$\xymatrix{
\Hom_{\widetilde\cC}(\II X,\II X)\ar[r]^{\nu_X}\ar[d]^{-\circ \II(\alpha)}& \Hom_{\cC}(X,\RR\circ\II X)\ar[d]^{-\circ\alpha}\\
\Hom_{\widetilde\cC}(\II Y,\II X)\ar[r]^{\nu_Y}& \Hom_{\cC}(Y,\RR\circ\II X)
}$$
commutes. This implies that~$\eta_X\circ\alpha=\nu_Y(\II(\alpha))$. Since $\II$ is faithful and $\nu_Y$ an isomorphism, we have $\eta_X\circ\alpha\not=0$ for any non-zero $\alpha$ and indeed $\eta_X$ is a monomorphism for all $X$.
\end{proof}

\subsection{Properties of Frobenius extensions}
Fix $\cC,\widetilde\cC$ as in Definition~\ref{DefFE}.

\begin{prop}\label{enough}
The category~$\widetilde\cC$ contains enough projective objects if and only if $\cC$ contains enough projective objects.
If there exist enough projective objects, then
\begin{enumerate}[(i)]
\item the projective object in $\cC$ are the direct summands of objects $\RR P$, for arbitrary~$P\in p\widetilde\cC$;
\item  the projective object in $\widetilde\cC$ are the direct summands of objects $\II P$ (or equivalently $\CC P$), for arbitrary~$P\in p\cC$.
\end{enumerate}
The same statements hold true for injective objects.
\end{prop}
\begin{proof}
Assume first that~$\cC$ contains enough projective objects. Consider an arbitrary~$N\in\widetilde\cC$ with epimorphism $\pi:P\tto \RR N$, for~$P$ projective, in~$\cC$. Define $\alpha:\II P\to N$ as the morphism corresponding to~$\pi$ under $(\II,\RR)$-adjunction. By Lemma~\ref{LemBasix}(ii),~$\II P$ is projective. By definition of adjunction, for any $\beta:N\to M$ in $\widetilde{\cC}$, we have a commuting square
$$\xymatrix{
\Hom_{\cC}(P,\RR N)\ar[r]\ar[d]^{\RR(\beta)\circ-}& \Hom_{\widetilde\cC}(\II P, N)\ar[d]^{\beta\circ-}\\
\Hom_{\cC}(P,\RR M)\ar[r]& \Hom_{\widetilde\cC}(\II P, M),
}$$
where the horizontal arrows are isomorphisms. Therefore, if~$\beta\circ\alpha=0$, then~$\RR(\beta)\circ\pi=0$. Since $\pi$ is an epimorphism the latter means $\RR(\beta)=0$ and by faithfulness of~$\RR$ thus finally $\beta=0$. Hence, $\alpha$ is an epimorphism
and~$\widetilde\cC$ contains enough projective objects.
 

Now assume that~$\widetilde{\cC}$ contains enough projective objects. For any $M\in \cC$, we have an epimorphism $\pi: P\tto \CC M$ with $P\in p\widetilde{\cC}$. Using adjunction one constructs, as above, an epimorphism $\RR P\tto M$, where $\RR P$ is projective by Lemma~\ref{LemBasix}(ii). Hence $\cC$ contains enough projective objects.

The other statements in the proposition follow from the above construction, the symmetry between projective and injective objects and Definition~\ref{DefFE}(iii).
\end{proof}

\begin{lemma}\label{LemTriv}
For any $X\in\cC$ and $Y\in\widetilde{\cC}$, we have
$$\pd_{\widetilde\cC}\II X\,=\,\pd_{\widetilde\cC}\CC X\,\le\, \pd_{\cC}X\qquad\mbox{and}\qquad \pd_{\cC}\RR Y\,\le\, \pd_{\widetilde\cC}Y.$$
The same statements hold true for injective dimensions.
\end{lemma}
\begin{proof}
This follows immediately from adjointness properties of exact functors, see e.g. \cite[Proposition~7]{Sigma} for the case without injective or projective objects.
\end{proof}

\begin{lemma}\label{idP}
For all projective objects $P\in\cC$ and $Q\in\widetilde\cC$, we have
$$\id_{\cC}P\;=\;\id_{\widetilde\cC}\II P\;=\;\id_{\widetilde\cC}\CC P\qquad\mbox{and}\qquad \id_{\widetilde\cC}Q\;=\;\id_{\cC}\RR Q.$$
The same holds for projective dimensions of injective objects.
\end{lemma}
\begin{proof}
By Lemma~\ref{LemBasix}(iv), we find that~$P$ is a direct summand of~$\RR\CC( P)$. Lemma~\ref{LemTriv} thus implies
$$\id_{\cC}P\;\ge\; \id_{\widetilde{\cC}}\CC (P)\;\ge\;\id_{\cC}\RR\CC (P)\;\ge\;\id_{\cC}P.$$
Similarly, one proves $\id_{\widetilde\cC}Q=\id_{\cC}\RR Q.$
\end{proof}

\begin{lemma}\label{LemPhiPsi}
Assume $\cC$ has enough projective objects and $\Phi$ is an auto-equivalence of~$\cD_{per}(\cC)$ with quasi-inverse $\Psi$. If there exist triangulated functors $\widetilde\Phi$ and $\widetilde\Psi$ on $\cD_{per}(\widetilde{\cC})$ which satisfy 
$$\widetilde{\Phi}\circ \II\;\cong\;\CC\circ\Phi\quad\mbox{and}\quad \widetilde{\Psi}\circ\CC\;\cong\;\II\circ\Psi,$$
then $\widetilde\Phi$ is an auto-equivalence of~$\cD_{per}(\cC)$ with quasi-inverse $\widetilde\Psi$.
\end{lemma}
\begin{proof}
By assumption, we have isomorphisms of functors on $\cD_{per}(\widetilde{\cC})$
$$\widetilde{\Psi}\circ\widetilde{\Phi}\circ\II\,\cong\,\II\quad\mbox{and}\quad\widetilde{\Phi}\circ\widetilde{\Psi}\circ\CC\,\cong\,\CC.$$
The conclusion thus follows from Proposition~\ref{enough}, since the full subcategory of objects $\II P$ with $P$ a projective object in $\cC$ generates $\cD_{per}(\widetilde{\cC})$ as a triangulated category.
\end{proof}

\subsection{Frobenius extensions of noetherian rings revisited}

\begin{thm}\label{ThmEquivFro}
For $S$ a noetherian ring, the Frobenius extensions of~$S${\rm -mod}, in the sense of Definition~\ref{DefFE}, are, up to equivalence, given by~$R${\rm-mod} for all $\alpha$-Frobenius extensions $R$ of~$S$ in the sense of Definition~\ref{DefFEring}.
\end{thm}
\begin{proof}
Consider an $\alpha$-Frobenius extension~$R$ of~$S$. The~$(R,S)$-bimodule isomorphism $R\cong \Hom_S(R,{}_\alpha S)$ and exactness of $\Hom_S(R,-)$ imply an isomorphism of functors
$$R\otimes_S-\;\cong\; \Hom_S(R,{}_\alpha S\otimes_S-)\,:\;\; S\mbox{-Mod}\;\to\;R\mbox{-Mod.}$$
The faithful functor~$\II:=R\otimes_S-$ always restricts to a functor from $R$-mod to $S$-mod. Since $R$ is a finite ring extension, the faithful functor~$\RR:= \res^R_S$ restricts to a functor from $R$-mod to~$S$-mod. By the above isomorphism, also $\CC:=\Hom_S(R,-)$ restricts to a functor between those categories. It follows that~$\RR: R\mbox{-mod}\to S\mbox{-mod}$  satisfies the conditions in Definition~\ref{DefFE}, where the auto-equivalence $\EE$ in condition (iii) is~${}_\alpha(-)={}_\alpha S\otimes_S-$.

 Now consider a Frobenius extension~$\widetilde{\cC}$ of $S$-mod in the sense of Definition~\ref{DefFE}. By Proposition~\ref{enough}, $\widetilde{\cC}$ contains enough projective objects and $p\widetilde{\cC}=\add(\II S)$. Define the ring and functor
$$R:=\End_{\widetilde{\cC}}(\II S)^{\op}\quad\mbox{ and}\qquad \Phi=\Hom_{\widetilde{\cC}}(\II S,-): \widetilde{\cC}\to R\mbox{-Mod}.$$
 By faithfulness of $\II: S^{\op}\hookrightarrow R^{\op}$, it follows that~$S$ is a subring of $R$. The commuting diagram
 $$\xymatrix{
 \Hom_{\widetilde{\cC}}(\II S,X)\ar[rr]^{\sim}\ar[d]^{-\circ\II(a)}&& \Hom_{S}(S,\RR X)\ar[d]^{-\circ a}\ar[rr]^{\sim}&&\RR X\ar[d]^{a\cdot-}\\
 \Hom_{\widetilde{\cC}}(\II S,X)\ar[rr]^{\sim}&& \Hom_{S}(S,\RR X)\ar[rr]^{\sim}&&\RR X\\
 }$$
of group homomorphisms for any $X\in\widetilde{\cC}$ and $a\in S^{\op}=\Hom_{S}(S,S)$, shows that we have an isomorphism of functors
\begin{equation}\label{eqRes}\res^R_S\circ \Phi\;\cong\; \RR\;:\;\widetilde{\cC}\to S\mbox{-mod}.\end{equation}
Evaluation on $\II S$ shows that the left $S$-module $R$ is isomorphic to~$\RR\II S$. By Lemma~\ref{LemBasix}(ii), the latter is a (finitely generated) projective $S$-module. In particular, $R$ is a finite ring extension of $S$ and thus noetherian. Equation~\eqref{eqRes} then further implies that~$\Phi$ actually restricts to a functor from $\widetilde{\cC}$ to~$R$-mod. By construction, this is an exact functor which restricts to an equivalence between the categories of projective modules. It thus follows that we have an equivalence of categories
$$\Phi=\Hom_{\widetilde{\cC}}(\II S,-):\;\;\widetilde{\cC}\;\;\tilde\to\;\; R\mbox{-mod.}$$

To conclude the proof, we now only need to establish the isomorphism of bimodules in \eqref{eqbimod}. Consider a quasi-inverse $\Psi$ of $\Phi$. By uniqueness of adjoint functors and~\eqref{eqRes}, we have
$$\II\cong\Psi\circ\ind^{R}_S\qquad\mbox{and}\qquad \CC\cong \Psi\circ\coind^R_S.$$
Consequently, we have
$$R\otimes_S-\;\cong\;\Hom_S(R,-)\circ \EE.$$
The isomorphism of bimodules \eqref{eqbimod} thus follows by observing that any auto-equivalence $\EE$ of $S$-mod is of the form ${}_\alpha$ for some automorphism $\alpha$.\end{proof}

Recall that we take the convention that `functors' between $\mk$-linear categories are $\mk$-linear
\begin{cor}
Fix a field $\mk$. The Frobenius extensions of~$\mk${\rm-mod} are, up to equivalence, the categories $A${\rm -mod}, for~$A$ a finite dimensional Frobenius $\mk$-algebra.
\end{cor}
\begin{proof} 
Since $\EE$ in Definition~\ref{DefFE} must be a $\mk$-linear auto-equivalence of~$\mk$-mod, it is isomorphic to the identity. The result then follows as a special case of Theorem~\ref{ThmEquivFro}, by Example~\ref{exAlgebra}. 
\end{proof}


\section{Naive Gorenstein homological algebra in abelian categories}\label{SecGHA}

Fix for the entire section an abelian category~$\cC$ which {\em contains enough projective objects}.
We generalise some well-established notions from ring theory, see e.g. \cite{AB, Avramov, Holm}. 
\subsection{Gorenstein projective objects } 

\begin{ddef}\label{defGP}A {\bf totally acyclic complex} $P_\bullet$ in~$p\cC$ is an exact sequence of objects in~$p\cC$ which is right $p\cC$-acyclical. An object $X$ in~$\cC$ is {\bf Gorenstein projective} if and only if it is isomorphic to a syzygy object of a totally acyclic complex in~$p\cC$. We denote denote the full subcategory of~$\cC$ of Gorenstein projective modules by~$gp\cC$. 
\end{ddef}
Clearly,~$p\cC$ is a subcategory of~$gp\cC$.
In the notation of \cite[Section~5]{AR} we have~$gp\cC:=\mathscr{X}_{p\cC}$.

\begin{prop}\label{Prop0inf}${}$
\begin{enumerate}[(i)]
\item We have
$gp\cC\subset {}^\perp p\cC={}^\perp p\cC^{(-)}.$ 
\item $gp\cC$ is the category of all $X$ which admit a projective coresolution with each cocycle in~${}^\perp p\cC$.
\item
 Each $X\in gp\cC$ satisfies either $\pd_{\cC}X=0$ or $\pd_{\cC}X=\infty$.
 \end{enumerate}
\end{prop}
\begin{proof}
The equality~${}^{\perp}p\cC={}^\perp p\cC^{(-)}$ is straightforward. Parts (i) and (ii) then follow immediately from Definition \ref{defGP} and Lemma~\ref{secPacyc}. 

To prove part~(iii) let $X$ be an object of projective dimension~$k\in\mZ_{>0}$ and~$N$ some object with $\Ext^k_{\cC}(X,N)\not=0$. For an epimorphism $P\tto N$, with $P$ projective, we find an exact sequence
$$\Ext^k_{\cC}(X,P)\to\Ext^k_{\cC}(X,N)\to0,$$
proving that~$X$ is not in~${}^\perp p\cC$. Part~(iii) thus follows from part~(i).
\end{proof}
An immediate consequence of Proposition~\ref{Prop0inf}(ii), or Definition~\ref{defGP}, is that
\begin{equation}\label{eqgpCin}gp\cC\;\subseteq\; \Omega^k\cC,\qquad\mbox{ for any~$k\in\mN$}.\end{equation}

The following is the analogue of~\cite[Proposition~5.1]{AR} or~\cite[Theorem~2.5]{Holm}.
\begin{prop}\label{PropRes}
The subcategory~$gp\cC$ is projectively resolving in~$\cC$.
\end{prop}
\begin{proof}
It is an easy exercise to prove that~${}^\perp p\cC$ is resolving, which we will use freely. Now consider a short exact sequence $X\hookrightarrow Y\tto Z$ with $Z\in gp\cC$.

Assume first that~$X\in gp\cC$. Clearly $Y\in {}^{\perp}p\cC$. Take $P$, resp. $Q$, to be the zero component of the projective coresolution of~$X$, resp. $Z$, in Proposition~\ref{Prop0inf}(ii) and denote the respective cokernels by~$X_1$ and~$Z_1$. Lemma~\ref{baby} then implies that there exists a short exact sequence
$$0\to Y\to P\oplus Q\to Y_1\to 0,$$
where $Y_1$ admits a short exact sequence $X_1\hookrightarrow Y_1\tto Z_1$ with again $X_1,Z_1\in gp\cC$. We can thus continue this procedure to construct a projective coresolution of~$Y$ where all cocycles are in~${}^{\perp}p\cC$. Again by Proposition~\ref{Prop0inf}(ii), we have~$Y\in gp\cC$.

Now assume that~$Y\in gp\cC$. Take $P$, resp. $Q$, to be the zero component of the projective coresolution of~$Y$, resp. $Z$, in Proposition~\ref{Prop0inf}(ii) and denote the respective cokernels by~$Y_1,Z_1\in gp\cC$. By Lemma~\ref{baby2}, there exists an exact sequence
$$0\to X\to P\to Q\oplus Y_1\to Z_1\to 0.$$
 Hence we have a short exact sequence $X\hookrightarrow P\tto X_1$, where $X_1$ admits a short exact sequence 
$$0\to X_1\to Q\oplus Y_1\to Z_1\to 0.$$
As the middle and right term in the above short exact sequence are again Gorenstein projective, we can proceed iteratively to construct a projective coresolution of~$X$ as in Proposition~\ref{Prop0inf}(ii).

Finally, if $X=X'\oplus X''\in gp\cC$, we claim that both $X'$ and~$X''$ are in $gp\cC$ as well. By Proposition~\ref{Prop0inf}(ii), there exist $P\in p\cC$ and $G\in gp\cC$ for which we have a short exact sequence $X\hookrightarrow P\tto G$. Hence, the pushout $X_1'= P\sqcup_{X}X''$ admits short exact sequences
$$0\to X'\to P\to X_1'\to 0\qquad\mbox{and}\qquad0\to X''\to X_1'\to G\to 0,$$
by \cite[Theorems~2.15 and~2.54]{Freyd}.
Taking the direct sum with $X'$ in the first two terms of the second exact sequence yields $X\hookrightarrow X_1'\oplus X'\tto G$.
As $X$ and~$G$ are in $gp\cC$, the first part of the proof shows that~$X'_1\oplus X'\in gp\cC$ and thus that~$X_1'$ is again the direct summand of a Gorenstein projective object. We can thus repeat the construction to construct a projective coresolution of~$X'$ as in Proposition~\ref{Prop0inf}(ii).
\end{proof}

\begin{lemma}\label{PropCharP}
If for~$G\in\ gp\cC$ there exists $q\in\mZ_{>0}$ such that~$\Ext_{\cC}^q(G',G)=0$ for all $G'\in gp\cC$, then $G$ is projective.
\end{lemma}
\begin{proof}
Assume first that~$q=1$. Definition~\ref{defGP} implies we have a short exact sequence
\begin{equation}\label{eqGPG}0\to G\to P\to G_1\to 0,\qquad\mbox{ for some $P\in p\cC$ and $G_1\in gp\cC$}.\end{equation}
By assumption, such a short exact sequence should split and hence both $G$ and~$G_1$ are projective.

Now assume that the statement has been proved for~$q_0\ge 1$ and consider the case $q=q_0+1$. We still have \eqref{eqGPG}. By Proposition~\ref{Prop0inf}(i), this implies that 
$$\Ext^{q_0}_{\cC}(G',G_1)\cong \Ext^{q_0+1}_{\cC}(G',G)=0,$$
for all $G'\in gp\cC$. Hence we can apply the induction step to conclude that~$G_1$ is projective. The short exact sequence~\eqref{eqGPG} then shows that also $G$ must be projective.
\end{proof}

\subsection{G-dimension}
We introduce the G-dimension of objects in an abelian category~$\cC$ following~\cite{AB}, \cite[Section~3]{Avramov} or \cite[Definition~2.8]{Holm}. 
\begin{ddef}\label{DefGD}
The G-dimension~$\Gd_{\cC}X\in\mN\cup\{\infty\}$ of $X\in \cC$ is the minimal length of a 
 $gp\cC$-resolution of $X$.\end{ddef}
Clearly,~$\Gd_{\cC}X=0$ if and only if $X\in gp\cC$. We use the notation~$gp\cC^{(n)}$, resp. $gp\cC^{(-)}$, for the full subcategory of all objects of G-dimension at most $n$, resp. finite G-dimension. 

The following is an analogue of \cite[Theorem~3.1]{Avramov} or \cite[Theorem~2.10]{Holm}. The proof is also identical, so we only give a sketch.
\begin{prop}\label{PropPPG}The following are equivalent, for~$k\in\mN$ and an object $X\in\cC$:
\begin{enumerate}[(i)]
\item $X$ has G-dimension at most $k$;
\item we have~$\Omega^k(P_\bullet)\in gp\cC$ for any projective resolution~$P_\bullet$ of~$X$;
\item $X$ admits a $gp\cC$-resolution of the form
\begin{equation}\label{strictres}0\to P_k\to P_{k-1}\to\cdots\to P_1\to G_0\to 0,\qquad\mbox{ with $P_i\in p\cC$, for~$1\le i\le k$.}\end{equation}
\end{enumerate}
\end{prop}
\begin{proof}
Clearly (ii) and (iii) imply (i). That~(i) implies (ii) is an application of Proposition~\ref{Prop0inf}(i) and the comparison lemma, mutatis mutandis the proof of \cite[Lemma~2.5]{Avramov}. That~(ii) implies (iii) is an application of the the comparison lemma, mutatis mutandis the proof of \cite[Theorem~2.10]{Holm}.
\end{proof}
This proposition has four useful corollaries.
\begin{cor}\label{CorKGX}\label{CorSpecial}
For any $X\in gp\cC^{(-)}$, there exist $G\in gp\cC$ and $K\in p\cC^{(-)}$ (more precisely with $\pd_{\cC}K=\Gd_{\cC}X-1$), with short exact sequence
\begin{equation}\label{KGX}0\to K\to G\to X\to 0.\end{equation}
In particular, each object in~$gp\cC^{(-)}$ admits a special right $gp\cC$-approximation.
\end{cor}
\begin{proof}
Proposition~\ref{PropPPG} implies that each $X\in gp\cC^{(-)}$ admits such a short exact sequence. By Proposition~\ref{Prop0inf}(i), $p\cC^{(-)}\subseteq gp\cC^{\perp_1}$, so $G\tto X$ is a special right $gp\cC$-approximation.
\end{proof}

\begin{cor}\label{CharGP}
We have~${}^\perp p\cC\,\cap\, gp\cC^{(-)}\;=\;gp\cC$.
\end{cor}
\begin{proof}
The inclusion~${}^\perp p\cC\cap gp\cC^{(-)}\supset gp\cC$ is clear. Now take $X\in gp\cC^{(-)}$, with short exact sequence~\eqref{KGX}.
If~$X\in {}^\perp p\cC={}^\perp p\cC^{(-)}$ then the sequence must split. Hence,~$X$ is a direct summand of $G\in gp\cC$, so by Proposition~\ref{PropRes} we have $X\in gp\cC$, which concludes the proof.
\end{proof}

The following corollary shows that the G-dimension is a refinement of the projective dimension. Definition~\ref{DefGD} thus allows to capture more information on objects of infinite projective dimension.
\begin{cor}\label{CorFinDim}
For an arbitrary object $X$ in~$\cC$ we have
\begin{enumerate}[(i)]
\item $\Gd_{\cC} X\;\le\; \pd_{\cC}X$;
\item $\Gd_{\cC} X\;=\; \pd_{\cC}X\;$ if $\;\pd_{\cC}X<\infty$.
\end{enumerate}
\end{cor}
\begin{proof}
Part (i) is immediate by Definition~\ref{DefGD}, as~$p\cC\subset gp\cC$. Now assume that~$\pd_{\cC}X=k_0\in\mN$, so $\Gd_{\cC}X=k\le k_0$ by part~(i). Consider the short exact sequence~\eqref{KGX}. As one can not have precisely one object of infinite projective dimension in a short exact sequence we have $\pd_{\cC}G<\infty$. Proposition~\ref{Prop0inf}(iii) implies that~$G\in p\cC$. But then we find $k_0=\pd_{\cC}X\le k$ and hence $k=k_0$.\end{proof}

The following is an analogue of \cite[Theorem~2.20]{Holm}.
\begin{cor}\label{CorGDExt}
For any~$X\in gp\cC^{(-)}$ and~$k\in\mN$, the following are equivalent:
\begin{enumerate}[(i)]
\item $\Gd_{\cC}X\le k$;
\item $\Ext_{\cC}^j(X,-)$ is trivial on~$p\cC$, for all $j> k$.
\end{enumerate}
\end{cor}
\begin{proof}
For $k=0$, the statement is a reformulation of Corollary~\ref{CharGP}. Now take $X\in gp\cC^{(-)}$ with $gp\cC$-resolution~$G_\bullet$ of length $\Gd_{\cC}(X)>0$. The syzygy$Y:=\Omega^1(G_\bullet)$ admits a short exact sequence
$ Y\hookrightarrow G_0\tto X.$
We clearly have $\Gd_{\cC}X=\Gd_{\cC}Y+1$ and
$$\Ext^j_{\cC}(X,-)\;\cong\;\Ext^{j-1}_{\cC}(Y,-)\quad\mbox{on}\; \,p\cC,\qquad \mbox{for~$j>0$.}$$
The full result thus follows by induction on $k$.
\end{proof}

\subsection{Gorenstein resolutions and extensions}
The following special case of Definition~\ref{properDef} follows \cite[Section~4]{Avramov}.
\begin{ddef}
A {\bf Gorenstein resolution} of an object $X\in \cC$ is a proper $gp\cC$-resolution of~$X$. Let~${}^g\cC$ denote the full subcategory of~$\cC$ of objects admitting a Gorenstein resolution.
\end{ddef}

\begin{lemma}\label{lemres}
A $gp\cC$-resolution as in equation~\eqref{strictres} is a Gorenstein resolution. Consequently we have~$gp\cC^{(-)}\,\subseteq {}^g\cC$. More precisely, for any~$k\in\mN$,~$gp\cC^{(k)}$ is the full subcategory of~${}^g\cC$ of all objects admitting Gorenstein resolutions of length $k$.
\end{lemma}
\begin{proof}
Since $\Hom_{\cC}(G,-)$ is exact on projective modules, for~$G\in gp\cC$, the fact that a $gp\cC$-resolution as in equation~\eqref{strictres} is a Gorenstein resolution is a standard exercise in homological algebra, see e.g. \cite[Lemma~4.1]{Avramov}. The other claims follow directly from Proposition~\ref{PropPPG}. 
\end{proof}

Following~\cite[Section~4]{Avramov} or \cite[Section~3]{Holm2}, we introduce the Gorenstein extension groups.
\begin{ddef}\label{defGE}
For any~$k\in\mN$, the functor~$\GE_{\cC}^k(-,-):({}^g\cC)^{\op}\times \cC\to {\mathbf{Ab}}$ is defined by
$$\GE_{\cC}^k(X,Y)\;\cong\; \HH^k(\Hom_{\cC}(G_\bullet, Y)),$$
for~$G_\bullet$ a Gorenstein resolution of~$X$.
\end{ddef}
\subsubsection{}\label{DefExtbla}That this is well-defined (does not depend on the choice of Gorenstein resolution) and natural in $X$ and $Y$ follows from the Comparison Lemma, see \cite[Theorem~4.2(1)]{Avramov}. As a standard essential property of relative homology, it also follows that, for~$X\in {}^g\cC$, and a {\em left $gp\cC$-acyclical} short exact sequence $A\hookrightarrow B\tto C$ in~$\cC$, there exists a long exact sequence
\begin{equation}\label{les}0\to \GE_{\cC}^0(X,A)\to\cdots\to \GE_{\cC}^i(X,B)\to  \GE_{\cC}^i(X,C)\to  \GE_{\cC}^{i+1}(X,A)\to\cdots,\end{equation}
see e.g.~\cite[Proposition~4.4]{Avramov}.


We note some immediate consequences of the definition.
Knowledge of a special $gp\cC$-approximation allows to calculate Gorenstein extension groups from ordinary ones.
\begin{lemma}\label{LemKGX}
Consider $X\in gp\cC^{(-)}$, with associated short exact sequence 
\eqref{KGX}. We have isomorphisms of functors
$$\GE^k_{\cC}(X,-)\cong\begin{cases}
\coker\left(\Hom_{\cC}(G,-)\to\Hom_{\cC}(K,-)\right)&\mbox{if $k=1$;}\\
\Ext^{k-1}_{\cC}(K,-)&\mbox{if $k>1$.}\end{cases}$$
\end{lemma}
\begin{proof}
We consider a finite projective resolution of $K$, which allows to construct a Gorenstein resolution of $X$. The isomorphisms then follow by definition.
\end{proof}

\begin{lemma}\label{LemExtProp}
We have isomorphisms of functors
\begin{enumerate}[(i)]
\item $\GE_{\cC}^0(-,-)\cong\Hom_{\cC}(-,-)$ on~${}^g\cC^{\op}\times \cC$;
\item $\GE_{\cC}^k(X,-)\cong\Ext^k_{\cC}(X,-)$ if $X\in p\cC^{(-)}$, for all $k\in\mN$;
\item $\GE_{\cC}^k(X,-)\cong\Ext^k_{\cC}(X,-)$ on~$p\cC^{(-)}$ if $X\in gp\cC^{(-)}$, for all $k\in\mN$.
\end{enumerate}
\end{lemma}
\begin{proof}
Part (i) is immediate by the left exactness of~$\Hom_{\cC}(-,Y)$. Part (ii) follows from taking a finite projective resolution of~$X$ as the Gorenstein resolution~$G_\bullet$.

We prove Part (iii) for $k>1$, the case $k=1$ being similar. We consider a short exact sequence~\eqref{KGX} and know by Lemma~\ref{LemKGX} that $\GE_{\cC}^k(X,-)\cong\Ext^{k-1}_{\cC}(K,-)$. Since~$\Ext^i(G,-)=0$ on $p\cC^{(-)}$ if $i>0$, we have $\Ext^{k-1}_{\cC}(K,-)\cong \Ext^k_{\cC}(X,-)$ on $p\cC^{(-)}$.
\end{proof}

The following is the analogue of \cite[Theorem~4.2(2)]{Avramov} in our setting.
\begin{prop}\label{PropGDExtg}
For any~$X\in{}^g\cC$ and~$k\in\mN$, the following are equivalent:
\begin{enumerate}[(i)]
\item $\Gd_{\cC}X\le k$;
\item $\GE_{\cC}^{j}(X,-)=0$ for all $j>k$;
\item $\GE_{\cC}^{k+1}(X,-)=0$.
\end{enumerate}
\end{prop}
\begin{proof}
The implications $(i)\Rightarrow (ii)\Rightarrow (iii)$ are immediate, by Lemma~\ref{lemres}. 

Now consider $X\in {}^g\cC$, satisfying $\GE_{\cC}^{k+1}(X,-)=0$, with Gorenstein resolution~$G_\bullet$. Define $Y:=\Omega^k(G_\bullet)$. As $G_\bullet$ is a Gorenstein resolution, the definition of~$\GE_{\cC}$ implies that
$$\GE_{\cC}^{1}(Y,-)\,\cong\, \GE_{\cC}^{k+1}(X,-)\,=\,0.$$
Now define $Z:= \Omega^{k+1}(G_\bullet)$, with inherited short exact sequence
\begin{equation}\label{sesZGk}0\to Z\to G_k\to Y\to 0.\end{equation}
As $G_\bullet$ is a Gorenstein resolution, this short exact sequence is left $gp\cC$-acyclical. We can thus apply equation~\eqref{les} which yields a short exact sequence
$$0\to \GE_{\cC}^0(Y,Z)\to \GE_{\cC}^0(Y,G_k)\to \GE_{\cC}^0(Y,Y)\to 0$$
By Lemma~\ref{LemExtProp}(i), we then find $\Hom_{\cC}^0(Y,G_k)\tto \Hom_{\cC}^0(Y,Y)$, which implies that the short exact sequence \eqref{sesZGk} splits, so $Y$ is Gorenstein projective. Hence, we have a finite $gp\cC$-resolution
 $$0\to Y\to G_{k-1}\to\cdots\to G_1\to G_0\to 0,$$
 of~$X$, so by definition~$\Gd_{\cC}X\le k$. This shows that~(iii) implies (i), concluding the proof.
 \end{proof}
 
 \begin{cor}\label{CorDS}
 For any $X,Y\in {}^g\cC$, we have $\Gd_{\cC}(X\oplus Y)=\max\{\Gd_{\cC}X,\Gd_{\cC}Y\}$.
 \end{cor}
 
By~\cite{Holm2}, Gorenstein extensions are consistent with respect to the concept of opposite categories. 
 \begin{lemma}\label{Extop}
Assume that~$\cC$ also has enough injective objects and take $M\in{}^g\cC$ and~$N\in{}^g(\cC^{\op})$.  We have isomorphisms of abelian groups
 $$\GE^j_{\cC}(M,N)\;\cong\; \GE^j_{\cC^{\op}}(N,M).$$
 \end{lemma}
 \begin{proof}
 Mutatis mutandis \cite[Theorem~3.6]{Holm2}.
 \end{proof}
 
For the following result we did not manage to find a reference.
\begin{prop}\label{LemExt1}
Consider $X\in {}^g\cC$ and $Y\in \cC$. To any non-zero element in the group $\GE_{\cC}^{1}(X,Y)$ we can associate $M\in\cC$ with left $gp\cC$-acylical non-split short exact sequence
$$0\to Y\to M\to X\to 0.$$
\end{prop}
\begin{proof}
Consider a Gorenstein resolution~$G_\bullet$ of~$X$ and set $N:=\Omega^1(G_\bullet)$. Hence we have a left $gp\cC$-acyclical short exact sequence
\begin{equation}
\label{eqiota}
\xymatrix{0\ar[r]& N\ar[r]^{\iota} &G_0\ar[r]& X\ar[r]& 0.}
\end{equation}
Definition~\ref{defGE} and elementary diagram chasing yields an exact sequence
\begin{equation}\label{eqch}\xymatrix{\Hom_{\cC}(G_0,Y)\ar[r]^{-\circ\iota}&\Hom_{\cC}(N,Y)\ar[r]& \GE_{\cC}^{1}(X,Y)\ar[r]& 0.}\end{equation}
To each non-zero element of~$\GE_{\cC}^{1}(X,Y)$, we can thus associate a morphism $\alpha: N\to Y$ which is not of the form $\beta\circ\iota$ for some morphism $\beta:G_0\to Y$. The two morphisms $\iota: N\to G_0$ and $\alpha: N\to Y$ define a pushout $M:= Y\sqcup_N G_0$, with commuting square
$$\xymatrix{ N\ar@{^{(}->}[r]^{\iota}\ar[d]^\alpha &G_0\ar[d]^{\gamma}\\\
Y\ar[r]^\beta&M,
}$$
with $\coker\beta\cong \coker\iota\cong X$, see the dual of \cite[Theorem~2.52]{Freyd}. Furthermore, as~$\iota$ is a monomorphism, so is~$\beta$, see \cite[Theorem~2.54]{Freyd}. We thus find a commuting diagram with exact rows
$$\xymatrix{0\ar[r]& N\ar[r]^{\iota}\ar[d]^\alpha &G_0\ar[d]^{\gamma}\ar[r]&X\ar[r]\ar@{=}[d]&0\\\
0\ar[r]&Y\ar[r]^\beta&M\ar[r]&X\ar[r]&0.
}$$
If the extension on the second row would split, it would follow that there is a morphism $\gamma':G_0\to Y$ such that~$\alpha=\gamma'\circ\iota$, which contradicts our assumptions.
It only remains to be proven that the short exact sequence on the second row is left $gp\cC$-acyclical. For any~$G\in gp\cC$, applying $\Hom_{\cC}(G,-)$ to the above diagram yields a commutative diagram
$$\xymatrix{\Hom_{\cC}(G,G_0)\ar[d]\ar[r]&\Hom_{\cC}(G,X)\ar@{=}[d]\\\
\Hom_{\cC}(G,M)\ar[r]&\Hom_{\cC}(G,X).
}$$
As \eqref{eqiota} is left $gp\cC$-acyclical, the top horizontal arrow is a group epimorphism, hence the lower horizontal arrow is also surjective. This concludes the proof.
\end{proof}

\subsection{Tate cohomology}
The following definition extends the one in \cite[Section~5]{Avramov}.
\begin{ddef}
For $X\in gp\cC^{(k)}$ with projective resolution~$Q_\bullet$, we have $G:=\Omega^k(Q_\bullet)\in gp\cC$ by Proposition~\ref{PropPPG}. We consider a totally acyclic complex $P_\bullet$ with $G=\Omega^k(P_\bullet)$. Then we define, for any $Y\in\cC$, the abelian groups
$$\widehat{\Ext}^i_{\cC}(X,Y)\;=\;\HH^i(\Hom_{\cC}(P_\bullet, Y)),\qquad\mbox{for all $i\in\mZ$.}$$
\end{ddef}
That this is well-defined (does not depend on the choice of $P_\bullet$ and $Q_\bullet$) and yields a functor
$$\widehat{\Ext}^i_{\cC}(-,-)\,:(gp\cC^{(-)})^{\op}\times\cC\to\Ab,$$
follows as in \cite[Section~5]{Avramov}.

\begin{rem}\label{RemTate} The following observations follow by definition.
\begin{enumerate}[(i)]
\item We have $\widehat{\Ext}^i_{\cC}(X,-)=0$, for all $i\in\mZ$, if $X\in p\cC^{(-)}$. 
\item If~$G\in gp\cC$, we have 
$$\widehat{\Ext}^i_{\cC}(G,-)\;\cong\; \Ext^i_{\cC}(G,-),\qquad\mbox{if $i>0$.}$$
\end{enumerate}
Lemma~\ref{secPacyc} and the fact ${}^\perp p\cC={}^{\perp} p\cC^{(-)}$ imply that also  $\widehat{\Ext}^i_{\cC}(-,Y)=0$, for all $i\in\mZ$, if $Y\in p\cC^{(-)}$.
\end{rem}

The following result can be proved as in \cite[Proposition~5.6]{Avramov}.
\begin{lemma}\label{LemAMTate}
For a short exact sequence $0\to X\to Y\to Z\to 0$ with objects in $gp\cC^{(-)}$, we have a long exact sequence of functors
$$\cdots\to \widehat{\Ext}^{i-1}_{\cC}(X,-)\to\widehat{\Ext}^{i}_{\cC}(Z,-)\to\widehat{\Ext}^{i}_{\cC}(Y,-)\to\widehat{\Ext}^{i}_{\cC}(X,-)\to\widehat{\Ext}^{i+1}_{\cC}(Z,-)\to\cdots$$
\end{lemma}

\begin{cor}\label{TateCor}
For $X\in gp\cC^{(-)}$ with short exact sequence \eqref{KGX}, we have
$$\widehat{\Ext}^i_{\cC}(X,-)\;\cong\;\Ext^i_{\cC}(G,-),\qquad\mbox{if $i>0$.}$$
\end{cor}
\begin{proof}
By Remark~\ref{RemTate}(i) and Lemma \ref{LemAMTate}, we find
$\widehat{\Ext}^j_{\cC}(X,-)\cong\widehat{\Ext}^j_{\cC}(G,-)$, for~$j\in\mZ$.
Remark~\ref{RemTate}(ii) thus concludes the proof.
\end{proof}

The main result in \cite{Avramov} states that, for noetherian rings, the functors $\Ext^i$, $\widehat{\Ext}^i$ and $\GE^i$ form a long exact sequence of bi-functors. The proof can be generalised to our setting of arbitrary abelian categories containing enough projective objects. If we are only interested in a functorial version, the above already implies the claim.
\begin{prop}\label{PropTateGExt}
For any $X\in gp\cC^{(-)}$, we have an exact sequence of functors
$$0\to \GE^1_{\cC}(X,-)\to \Ext^1_{\cC}(X,-)\to \widehat{\Ext}^1_{\cC}(X,-)\to \GE^2_{\cC}(X,-)\to \Ext^2_{\cC}(X,-)\to\cdots.$$
\end{prop}
\begin{proof}
Consider short exact sequence \eqref{KGX}. 
For $K:=\coker(\Hom_{\cC}(G,-)\to\Hom_{\cC}(K,-))$, we have a long exact sequence
$$0\to K\to \Ext^1_{\cC}(X,-)\to \Ext^1_{\cC}(G,-)\to\Ext^1_{\cC}(K,-)\to \Ext^2_{\cC}(X,-)\to\cdots.$$
Inserting the natural isomorphisms in Lemma~\ref{LemKGX} and Corollary~\ref{TateCor} concludes the proof.
\end{proof}

\subsection{Naively Gorenstein categories}
The following definition is inspired by \cite[Section~3]{Avramov}.
\begin{ddef}\label{DefGorCat2}The category~$\cC$ is {\bf naively $d$-Gorenstein}, for~$d\in\mN$, if $\cC =gp\cC^{(d)}$. The category~$\cC$ is {\bf weakly Gorenstein} if $gp\cC={}^{\perp}p\cC$.
\end{ddef}
\begin{lemma}\label{DefOm}
$\cC$ is naively $d$-Gorenstein if and only if $gp\cC=\Omega^d\cC$
\end{lemma}
\begin{proof}
By equation~\eqref{eqgpCin}, it suffices to prove that~$\cC =gp\cC^{(d)}$ if and only if $\Omega^d\cC\subseteq gp\cC$. This follows from Proposition~\ref{PropPPG}.
\end{proof}

\begin{prop}\label{ContFin}
Let $\cC$ be a naively Gorenstein category.
\begin{enumerate}[(i)]
\item The subcategory~$gp\cC$ is contravariantly finite. Furthermore, every object in~$\cC$ admits a special right $gp\cC$-approximation.
\item The pair $(gp\cC,p\cC^{(-)})$ is a hereditary cotorsion pair in~$\cC$ which admits enough projectives.\end{enumerate}
\end{prop}
\begin{proof}
Corollary~\ref{CorSpecial} implies that any object in~$\cC$ admits a special right $gp\cC$-approximation. In particular,~$gp\cC$ is contravariantly finite, proving part~(i).

Now we prove part~(ii). By Proposition~\ref{Prop0inf}(i), the functors~$\Ext^j_{\cC}(-,-)$ vanish on~$gp\cC^{\op}\times p\cC^{(-)}$ for~$j>0$. We thus have a hereditary torsion pair (with enough projective objects, by part~(i)) if~$gp\cC^{\perp_1}\subseteq p\cC^{(-)}$ and~${}^{\perp_1}p\cC^{(-)}\subseteq gp\cC$.

Consider $X\in gp\cC^{\perp_1}$. Since $\cC=gp\cC^{(-)}$, Corollary~\ref{CorKGX} implies we have a short exact sequence
$$0\to K_0\to G_0\to X\to 0,$$
with $G_0\in gp\cC$ and~$\pd_{\cC}K_0<\infty$, so in particular $K_0\in gp\cC^{\perp}$. By assumption on~$X$ it thus follows that~$G_0\in gp\cC^{\perp_1}$. Lemma~\ref{PropCharP} thus implies that~$G_0$ is projective and thus that~$X\in p\cC^{(-)}$.

Now consider $Y\in {}^{\perp_1}p\cC^{(-)}$. We consider again a short exact sequence
$$0\to K_1\to G_1\to Y\to 0,$$
with $G_1\in gp\cC$ and~$K_1\in p\cC^{(-)}$. By assumption on~$Y$, this extension must vanish and~$Y$ is a direct summand of~$G_1$, so $Y\in gp\cC$.
\end{proof}


\begin{prop}\label{Prop3}
For $\cC$ naively Gorenstein, $X\in \cC$ and~$k\in\mN$, the following are equivalent:
\begin{enumerate}[(i)]
\item $\Gd_{\cC}X\le k$;
\item $\Ext_{\cC}^j(X,-)$ is trivial on~$p\cC$, for all $j> k$.
\item $\GE_{\cC}^j(X,-)=0$, for all $j> k$.
\end{enumerate}
\end{prop}
\begin{proof}
Since~$gp\cC^{(-)}={}^g\cC=\cC$, the statements follow from Corollary~\ref{CorGDExt} and Proposition~\ref{PropGDExtg}. \end{proof}

\begin{cor}\label{Remperp1}
If~$\cC$ is naively Gorenstein, it is also weakly Gorenstein.
\end{cor}

\subsection{Iwanaga-Gorenstein properties}
Inspired by \cite{Iwanaga} or \cite[Section~6]{AR}, we introduce the following definition, for which we use the first letters of Iwanaga and Gorenstein.
\begin{ddef}\label{DefGorCatW}We say the category~$\cC$ is {$d$-$\cIG$}, for~$d\in\mN$, if the injective dimension of objects in~$p\cC$ bounded by~$d$. If~$\cC$ has enough injective objects, it is $d$-$\cGI$ if the projective dimensions of objects in $i\cC$ is bounded by $d$.
\end{ddef}

\begin{rem}
The ``Gorenstein symmetry conjecture'' asks whether, for $A$ a finite dimensional $\mk$-algebra, $A$-mod is $d$-$\cIG$ if only if it is $d$-$\cGI$.
\end{rem}

\begin{ex}
\label{ExTriv} Any abelian category~$\cC$ which admits enough projective and injective objects and has~$\gd\cC=d$ is $d$-$\cIG$ and $d$-$\cGI$.
\end{ex}

\begin{thm}\label{ThmFro}
Let $\widetilde\cC$ a Frobenius extension of~$\cC$. Then $\widetilde\cC$ is $d$-$\cIG$ (resp. $d$-$\cGI$) if and only if $\cC$ is $d$-$\cIG$ (resp. $d$-$\cGI$).
\end{thm}
\begin{proof}
This is an immediate consequence of Proposition~\ref{enough} and Lemma~\ref{idP}.
\end{proof}

\begin{lemma}\label{LemOm}
If~$\cC$ is $d$-$\cIG$, then we have~$\Omega^d\cC\subset{}^{\perp}p\cC$.
\end{lemma}
\begin{proof}
Take an arbitrary~$X\in\cC$, with projective resolution~$P_\bullet$. For $i>0$ and~$Q\in p\cC$, we have
$$\Ext^i_{\cC}(\Omega^d(P_\bullet),Q)\;\cong\; \Ext^{i+1}_{\cC}(\Omega^{d-1}(P_\bullet),Q)\;\cong\;\cdots\;\cong\;\Ext^{i+d}_{\cC}(X,Q).$$
However, as~$\id_{\cC} Q\le d<i+d$, the extension groups vanish and~$\Omega^d(P_\bullet)\in {}^{\perp}p\cC$.
\end{proof}

\begin{cor}\label{CorCor}
If~$\cC$ is $d$-$\cIG$, $gp\cC$ is the category of objects which admit a projective coresolution.
\end{cor}
\begin{proof}
Let $X$ have a projective coresolution~$Q^\bullet$. Clearly $X$ and each cocycle in~$Q^\bullet$ belong to~$\Omega^k\cC$ for any~$k\in\mN$. Lemma~\ref{LemOm} thus implies that all these objects are in~${}^{\perp}p\cC$. The conclusion follows from Proposition~\ref{Prop0inf}(ii).
\end{proof}

\begin{lemma}\label{LemAbed}
If~$\cC$ is naively $d$-Gorenstein, it is $d$-$\cIG$. If furthermore $\cC$ contains enough injective objects, then it is also $d$-$\cGI$.
\end{lemma}
\begin{proof}
For any~$Q\in p\cC$,~$i>0$ and~$X\in \cC$ with projective resolution~$P_\bullet$, we have $\Ext_{\cC}^{d+i}(X,Q)\cong\Ext^{i}_{\cC}(\Omega^d(P_\bullet),Q).$
By Lemma~\ref{DefOm},~$\Omega^d(P_\bullet)\in gp\cC$, so the extension groups must vanish. Hence we find $\id_{\cC}Q\le d$ and $\cC$ is $d$-$\cIG$.

Now assume $\cC$ contains enough injective objects and consider $I\in i\cC$, with projective resolution~$R_\bullet$. For any $G\in gp\cC$, we have
$$0=\Ext^1_{\cC}(G,I)\cong\Ext^{d+1}(G,\Omega^d(P_\bullet)).$$ 
By Lemma~\ref{DefOm}, $\Omega^d(P_\bullet)\in gp\cC$, so by Lemma~\ref{PropCharP} we have $\Omega^d(P_\bullet)\in p\cC$. This means that $\pd_{\cC}I\le d$ and $\cC$ is $d$-$\cGI$.
\end{proof}

\begin{lemma}\label{LemGG2}
The following are equivalent:
\begin{enumerate}[(i)]
\item $\cC$ is naively $d$-Gorenstein;
\item $\cC$ is weakly Gorenstein and $d$-$\cIG$.
\end{enumerate}
\end{lemma}
\begin{proof}
The implication (i)$\Rightarrow$(ii) follows from Corollary~\ref{Remperp1} and Lemma~\ref{LemAbed}.

Now assume that~$\cC$ is $d$-$\cIG$ with $gp\cC={}^\perp p\cC$. Lemma~\ref{LemOm} and equation~\eqref{eqgpCin} thus imply that~$\Omega^d\cC= gp\cC$.
The implication (ii)$\Rightarrow$(i) thus follows from Lemma~\ref{DefOm}.
\end{proof}

\begin{lemma}\label{pdidIG}
If~$\cC$ is $d$-$\cIG$, then $\pd_{\cC}X<\infty$ implies that $\id_{\cC}X\le d$.
\end{lemma}
\begin{proof}
If $X$ has a finite projective resolution, one can calculate $\Ext^{d+i}(-,X)=0$, for all $i>0$.
\end{proof}

\begin{cor}\label{pdidIG2}
If~$\cC$ is $d$-$\cIG$ and $d$-$\cGI$, then $\fd\cC\le d$.
\end{cor}


\section{Idempotent noetherian rngs}\label{SecIdem}
In this section, $R$ is a ``rng", a ring which need not contain a multiplicative identity. Modules are defined as for rings, omitting the condition that the identity acts as the identity morphism.

\subsection{Definitions}

A rng $R$ {\bf has enough idempotents}, if there exists a set $\cmE$ of idempotents in~$R$, such that we have direct sums of abelian groups
$$R\;=\;\bigoplus_{e\in \cmE} eR\;=\;\bigoplus_{e\in \cmE}Re.$$
Clearly, such~$R$ is unital ({\it i.e.} a ring) if and only if $\sharp\cmE<\infty$.
Consequently, if $R$ would be noetherian in the traditional sense, this would force $R$ to be a unital ring. We thus need an adapted definition.
\begin{ddef}
An {\bf idempotent noetherian rng} is a rng with enough idempotents such that~$Re$, resp. $eR$, is a noetherian left, resp. right,~$R$-module, for each $e\in\cmE$.
\end{ddef}
Clearly, $R$ is idempotent noetherian, if and only if $R^{\op}$ is.

\begin{ddef}\label{defRmod}
For a rng with enough idempotents, let $R$-mod denote the full category of all noetherian left modules $M$ which satisfy $M=\bigoplus_{e\in \cmE}eM$.
\end{ddef}

If~$R$ is a noetherian unital ring,~$R$-mod as defined above clearly corresponds to the abelian category of finitely generated (unital) modules. 
\begin{lemma}
If~$R$ is idempotent noetherian, $R$-mod is abelian and contains enough projective objects. The objects in~$R$-mod are the quotients of~$\bigoplus_{e\in \cmS}Re$, for~$\cmS$ finite multisets of elements in $\cmE$.
\end{lemma}
\begin{proof}
The category of noetherian modules is a Serre subcategory of the category of all modules and hence abelian. The condition on a module~$M$ to be noetherian also implies that it must be a quotient of $\bigoplus_{e\in \cmS}Re$, for~$\cmS$ finite. That any such quotient is noetherian follows again from the fact that $R$-mod is a Serre subcategory. It is then obvious that~$R$-mod contains enough projective modules since any module~$\bigoplus_{e\in \cmS}Re$ will be projective by construction.
\end{proof}
We also introduce the following notation:
\begin{itemize}
\item $\cP=\cP_{R}:=p(R\mbox{-mod})$ and $\cP^{\circ}=\cP^{\circ}_R:=\cP_{R^{\op}}$;
\item $\cP^{(-)}=\cP^{(-)}_R:=p(R\mbox{-mod})^{(-)}$;
\item $\cGP=\cGP_R:=gp(R\mbox{-mod})$ and $\cGP^{\circ}=\cGP^{\circ}_R:=\cGP_{R^{\op}}$.
\end{itemize}

\begin{ex}\label{exAR}
For a field $\mk$, we have the following special case of idempotent noetherian rngs.
A $\mk$-algebra is {\bf strongly locally finite}, see e.g. \cite[Section~2]{Serre}, if it admits a collection of idempotents~$\cmE$, such that 
$$A=\bigoplus_{e,f\in\cmE}eAf,\qquad\mbox{with }\;\dim_{\mk}eA<\infty\quad\mbox{and}\quad\dim_{\mk}Ae<\infty,\;\;\mbox{for each $e\in\cmE$}.$$
In this case, $A$-mod is the category of finite dimensional modules $M$ which satisfy $M=\bigoplus_{e\in\cmE}eM$.
\end{ex}

\subsection{Dualisation functors}\label{dualisation}
Let $R$ be a idempotent noetherian rng.

\begin{ddef}
The dualisation functors are the contravariant functors given by
$$\sF(-)=\bigoplus_{e\in\cmE}\Hom_R(-,Re):\;R\mbox{-mod}\,\to\,R^{\op}\mbox{-mod},$$
$$\sF^{\circ}(-)=\bigoplus_{e\in\cmE}\Hom_R(-,R^{\op}e):\;R^{\op}\mbox{-mod}\,\to\,R\mbox{-mod}.$$
\end{ddef}

The following lemma is trivial.
\begin{lemma}\label{LemFProj}
The functor~$\sF$ yields a contravariant equivalence $\cP\to\cP^{\circ}$ with inverse $\sF^{\circ}$.
\end{lemma}
\begin{prop}\label{PropFGP}
The functor~$\sF$ yields a contravariant equivalence $\cGP\to\cGP^{\circ}$ with inverse $\sF^{\circ}$.
\end{prop}
\begin{proof}
Consider $G\in\cGP$ and a totally acyclic complex~$P_\bullet$ in~$\cP$ with $G=\Omega^0(P_\bullet)$. As the complex is totally acyclic,~$\FF$ maps this sequence to an exact cocomplex~$Q^\bullet$ of projective modules $Q^k=\FF(P_k)\in\cP^{\circ}$. Since $\FF$ is contravariant and left exact, we have~$\FF(G)=\FF(\Omega^0(P_\bullet))\cong \CC^0(Q^\bullet)$. 
By Lemma~\ref{LemFProj}, applying $\FF^{\circ}$ to~$Q^\bullet$ will yield the original sequence $P_\bullet$, which is in particular exact. This thus implies that~$Q^\bullet$ is totally acyclic. Hence, we have~$\FF(G)\in \cGP^{\circ}$ and 
$$\FF^{\circ}\FF(G)\,\cong\, \FF^{\circ}\FF(\Omega^0(P_\bullet))\,\cong\, \Omega^0(P_\bullet)\,\cong\, G.$$
That $\FF$ and $\FF^{\circ}$ actually form mutual inverses then follows easily from Lemma~\ref{LemFProj}.\end{proof}

\subsection{Idempotent (Iwanaga-)Gorenstein rngs}
As a generalisation of \cite{Iwanaga} or \cite[Section~6]{AR}, we use the following definition. 
\begin{ddef}\label{DefIGRing}
An idempotent noetherian rng $R$ is {\bf idempotent $d$-Iwanaga-Gorenstein}, for~$d\in \mN$, if and only if both $R$-mod and $R^{\op}$-mod are $d$-$\cIG$.\end{ddef}

The following definition is a generalisation of the one in~\cite[Section~3]{Avramov}.
\begin{ddef}\label{DefGRing}
An idempotent noetherian rng $R$ is an {\bf idempotent $d$-Gorenstein rng} if both $R$-mod and $R^{\op}$-mod are naively $d$-Gorenstein.\end{ddef}

The following is the natural analogue of \cite[Theorem~3.2]{Avramov}.
\begin{thm}\label{Thmflp}
For an idempotent noetherian rng $R$ and $d\in\mN$, the following are equivalent:
\begin{enumerate}[(i)]
\item $R$ is idempotent $d$-Iwanaga-Gorenstein;
\item $R$ is idempotent $d$-Gorenstein.
\end{enumerate}
\end{thm}
\begin{proof}
The implication (ii)$\Rightarrow$(i) follows from Lemma~\ref{LemAbed}.
Now assume that~$R$ is idempotent $d$-Iwanaga-Gorenstein. We will prove that~$\cGP_R={}^{\perp}\cP_R$, i.e. that $R$-mod is naively Gorenstein. The conclusion then follows from Lemma~\ref{LemGG2} and symmetry between $R$ and $R^{\op}$.

The inclusion~$\cGP_R\subset{}^{\perp}\cP_R$ is given in Proposition~\ref{Prop0inf}(i). 
Consider $X\in {}^{\perp}\cP_R$ with projective resolution~$P_\bullet$. Hence $P_\bullet$ must be right $\cP$-acyclical. Consequently, Lemma~\ref{LemFProj} implies that~$\FF(X)$ has a projective coresolution~$Q^\bullet=\FF(P_\bullet)$. Corollary~\ref{CorCor} thus implies $\FF(X)\in \cGP^{\circ}$. Proposition~\ref{PropFGP} then states that~$X\in \cGP$, so we have indeed~$\cGP={}^\perp \cP$. 
\end{proof}

\section{Locally finite abelian categories with enough projective and injective objects}\label{SecLFP} Fix an arbitrary field $\mk$.
We introduce a class of categories which behave well with respect to Frobenius extension and Gorenstein homological algebra.

\subsection{lfp categories}\label{lfpcat} \begin{ddef}\label{Deflfp}
An {\bf \lfp  category~$\cC$} is a $\mk$-linear abelian category such that
\begin{enumerate}[(i)]
\item every object has finite length;
\item all $\Hom_{\cC}$-spaces are finite dimensional;
\item $\cC$ contains enough projective and injective objects.
\end{enumerate}
\end{ddef}

\begin{rem}\label{Remlfp}Consider Definition~\ref{Deflfp}.
\begin{enumerate}[(a)]
\item Conditions (i)-(ii) mean that~$\cC$ is a {\bf locally finite abelian category}.
\item An \lfp category is Krull-Schmidt, see e.g.~\cite[Section~5]{Krause}.
\item Every simple object in an \lfp category admits a projective cover, which follows from (b) and \cite[Lemma~3.6]{Krause}. Furthermore, any projective object is a direct sum of such covers.
\item If~$\mk$ is algebraically closed, we can omit condition (ii) by Schur's lemma and condition~(i).
\item $\cC$ is \lfp if and only if $\cC^{\op}$ is \lfpp.
\end{enumerate}
\end{rem}

Let $\Lambda$ denote a labelling set for the isoclasses of simple objects in $\cC$.  By Remark~\ref{Remlfp}(c), we have corresponding projective covers $\{P_\lambda\,|\,\lambda\in\Lambda\}$.
We define the $\mk$-algebra
\begin{equation}\label{DefA}A\,:=\,\bigoplus_{\lambda,\mu\in \Lambda}\Hom_{\cC}(P_\lambda,P_\mu)\;=\;\bigoplus_{\lambda,\mu}e_\lambda Ae_\mu,\end{equation}
with multiplication given by~$\alpha\beta=\beta\circ\alpha$ and with~$e_\lambda\in A$ the identity morphism of~$P_\lambda$.
By Definition~\ref{Deflfp} and Remark~\ref{Remlfp}(e), both $Ae_\lambda$ and $e_\lambda A$ are finite dimensional. Hence, A is strongly locally finite, see Example~\ref{exAR}. 
There is a well-known equivalence
\begin{equation}\label{PhiA}\Phi:=\bigoplus_{\lambda}\Hom_{\cC}(P_\lambda,-):\;\,\cC\,\;\,\tilde\to\,\;\, A\mbox{-mod}.\end{equation}
Indeed, $\Phi$ is exact and restricts to an equivalence between the categories of projective modules, which allows to quickly demonstrate the above equivalence.

\subsection{Gorenstein homological algebra in lfp categories}

\begin{ddef}\label{DefGorlfp}
Consider an \lfp category~$\cC$ and $d\in\mN$.
\begin{enumerate}[(i)]
\item $\cC$ is {\bf $d$-Iwanaga-Gorenstein} if $\cC$ is $d$-$\cIG$ and $d$-$\cGI$.
\item  $\cC$ is {\bf $d$-Gorenstein} if it is naively $d$-Gorenstein.
\end{enumerate}
\end{ddef}

\begin{rem}\label{remCIG}
Since we have a canonical equivalence $A^{\op}\mbox{-mod}\cong\cC^{\op}$, for $A$ in \ref{DefA}, $\cC$ is $d$-Iwanaga-Gorenstein if and only if $A$ is idempotent $d$-Iwanaga-Gorenstein.
\end{rem}

\begin{thm}\label{Corlfp}
For an \lfp category~$\cC$ and $d\in\mN$, the following are equivalent:
\begin{enumerate}[(i)]
\item $\cC$ is $d$-Iwanaga-Gorenstein;
\item $\cC$ is $d$-Gorenstein.
\end{enumerate}
\end{thm}
\begin{proof}
The implication (ii)$\Rightarrow$(i) is just Lemma~\ref{LemAbed}.

If $\cC$ is $d$-Iwanaga-Gorenstein, Remark~\ref{remCIG} and Theorem~\ref{Thmflp} show that $\cC$ (and also $\cC^{\op}$) is naively Gorenstein, so in particular we find (i)$\Rightarrow$(ii).
\end{proof}

\begin{rem}
Although not obvious from Definitions~\ref{DefGorlfp} and \ref{DefGRing}, it follows from Theorems~\ref{Thmflp} and~\ref{Corlfp} that the \lfp category~$\cC$ is Gorenstein if and only if $A$ in \eqref{DefA} is idempotent Gorenstein.
\end{rem}

\begin{lemma}
If the \lfp category~$\cC$ is $d$-Gorenstein, then $\fd\cC\le d$.
\end{lemma}
\begin{proof}
This follows from Theorem~\ref{Corlfp} and Corollary~\ref{pdidIG2}.
\end{proof}

\subsection{Frobenius extensions of lfp categories}

\begin{prop}\label{PropFElfp}
Consider a Frobenius extension~$\RR:\widetilde{\cC}\to\cC$ of $\mk$-linear abelian categories.
\begin{enumerate}[(i)]
\item $\widetilde{\cC}$ is \lfp if and only if $\cC$ is \lfpp.
\item Assuming $\cC$ and $\widetilde{\cC}$ are \lfpp, we have that~$\widetilde\cC$ is $d$-Gorenstein if and only if $\cC$ is so.
\end{enumerate}
\end{prop}
\begin{proof}
For part~(i), assume first that~$\cC$ is \lfpp. The faithfulness and exactness of~$\RR:\widetilde{\cC}\to\cC$ imply that conditions (i) and (ii) in Definition~\ref{Deflfp} are inherited by~$\widetilde{\cC}$ from $\cC$. That condition \ref{Deflfp}(iii) is also inherited follows from Proposition~\ref{enough}. The other direction follows similarly.

Part (ii) follows from Theorems~\ref{ThmFro} and~\ref{Corlfp}.
\end{proof}

For the remainder of this section, fix a Frobenius extension~$\RR:\widetilde{\cC}\to\cC$ of \lfp categories.  
\begin{prop}\label{PropGoGo1}
Assume $\cC$ (and hence $\widetilde\cC$) is Gorenstein. The Gorenstein projective modules in~$\widetilde\cC$ are precisely those $X\in\widetilde\cC$ for which $\RR(X) $ is Gorenstein projective in~$\cC$. More generally:
\begin{enumerate}[(i)]
\item $\Gd_{\widetilde\cC}X\;=\;\Gd_{\cC}\RR X,$ for any~$X\in\widetilde\cC$;
\item $\pd_{\widetilde\cC}X\;=\;\pd_{\cC}\RR X,$ if $\pd_{\widetilde\cC}X<\infty$;
\item $\Gd_{\widetilde\cC}\II Y\;=\;\Gd_{\cC} Y,$ for any~$Y\in\cC$;
\item $\pd_{\widetilde\cC}\II Y\;=\;\pd_{\cC}Y,$ if $\pd_{\cC}Y<\infty$.
\end{enumerate}
\end{prop}
\begin{proof}
First we prove part~(i).
By Propositions~\ref{enough}(ii) and~\ref{Prop3}, the objects $X\in\widetilde\cC$ satisfying~$\Gd_{\widetilde\cC}X\le k$ are precisely those which satisfy
$$\Ext^j_{\widetilde\cC}(X,\CC P)\;\cong\; \Ext^j_{\cC}(\RR X, P)\,=\,0,$$
for all $P\in p\cC$ and~$j>k$. This proves part~(i)

To prove part~(ii), now assume that~$\pd_{\widetilde\cC}X<\infty$. By Lemma~\ref{LemTriv}, we also have~$\pd_{\cC}\RR X<\infty$. The result then follows from part~(i) and Corollary~\ref{CorFinDim}(ii).

Part (iii) follows similarly as part~(i) and part~(iv) follows similarly as part~(ii).\end{proof}

We list crucial differences and similarities between projective and Gorenstein projective objects.
\begin{scholium}{\rm
Let $\RR:\widetilde\cC\to\cC$ be a Frobenius extension of Gorenstein \lfp categories.
\begin{enumerate}[(i)]
\item If~$M$ is (Gorenstein) projective in~$\cC$, then $\II M$ is (Gorenstein) projective in~$\widetilde\cC$.
\item Any projective object in~$\widetilde\cC$ is a direct summand of an $\II P$ with $P\in p\cC$. The corresponding statement is {\bf not} true for Gorenstein projective objects.
\item If~$N$ is (Gorenstein) projective in~$\widetilde\cC$, then $\RR N$ is (Gorenstein) projective in~$\cC$.
\item That $\RR N$ is Gorenstein projective in~$\cC$ is enough to conclude that~$N$ is Gorenstein projective in~$\widetilde\cC$. The corresponding statement is {\bf not} true for projective objects.
\end{enumerate}}
\end{scholium}

We state explicitly the special case of Proposition~\ref{PropGoGo1} where $\cC$ is as in Example~\ref{ExTriv}.
\begin{prop}\label{PropGdpd}
If~$\gd\cC=d$, then $\widetilde\cC$ is $d$-Gorenstein and~$X\in gp\widetilde\cC$ if and only if $\RR X\in p\cC$. More generally, we have
$$\Gd_{\widetilde\cC}X\;=\;\pd_{\cC}\RR X\quad\mbox{and}\quad \pd_{\widetilde\cC}\II Y=\pd_{\cC}Y,\quad\mbox{for any~$X\in\widetilde\cC$ and $Y\in\cC$.}$$
\end{prop}

\subsection{Nakayama and Serre functors versus Frobenius extensions}
\subsubsection{}
For a strongly locally finite algebra $A$, the {\bf Nakayama functor} 
$$\NN:\;A\mbox{-mod}\,\to\,A\mbox{-mod},$$ see e.g.~\cite[Section~2.3]{Serre}, is the composition of the dualisation functor~$\FF$ of Section~\ref{dualisation} with the ordinary duality~$\ast=\Hom_{\mk}(-,\mk)$. The resulting covariant functor~$\NN$ is thus right exact.

We recall a result of Mazorchuk and Miemietz:
\begin{lemma}\label{LemMM}\cite[Proposition~2.2]{Serre}.
For an \lfp category~$\cC$ which is $d$-$\cGI$, the derived Nakayama functor~$\cL\NN$ restricts to a Serre functor on~$\cD_{per}(\cC)$.
\end{lemma}
Consider a Frobenius extension~$\RR:\widetilde{\cC}\to\cC$, of  an \lfp category~$\cC$ which is $d$-$\cGI$. By Theorem~\ref{ThmFro} and Proposition~\ref{PropFElfp}(i), also $\widetilde\cC$ an \lfp category which is $d$-$\cGI$.
By Lemma~\ref{LemMM}, the categories $\cD_{per}(\widetilde\cC)$ and $\cD_{per}(\cC)$ admit Serre functors.

\begin{prop}\label{PropSerreIC}
Consider a Frobenius extension~$\RR:\widetilde{\cC}\to\cC$, of an \lfp category~$\cC$ which is $d$-$\cGI$. and let $\Phi$ denote the Serre functor of~$\cD_{per}(\cC)$. The Serre functor of~$\cD_{per}(\widetilde\cC)$ is the unique (up to isomorphism) auto-equivalence $\widetilde{\Phi}$ satisfying
$\widetilde{\Phi}\circ \II\;\cong\; \CC\circ \Phi.$
\end{prop}
\begin{proof}
Note that the functors~$\RR,\II,\CC$ lead to adjoint pairs of functors between the categories $\widetilde\cD:=\cD_{per}(\widetilde\cC)$ and $\cD:=\cD_{per}(\cC)$, by Lemma~\ref{LemBasix}.
For any $X\in\widetilde\cD$ and $Y\in\cD$, we have an isomorphism
\begin{equation}\label{HHHH}\Hom_{\widetilde\cD}(X,\CC\Phi Y)\cong \Hom_{\cD}(\RR X,\Phi Y)\cong\Hom_{\cD}(Y,\RR X)^\ast\cong \Hom_{\widetilde{\cD}}(\II Y,X)^\ast,\end{equation}
natural in $X$ and $Y$. 

For an auto-equivalence $\Psi$ of~$\widetilde\cD$ which satisfies $\Psi\II\cong \CC\Phi$, equation~\eqref{HHHH} implies isomorphisms
$$\kappa_{X,Y}:\; \Hom_{\widetilde\cD}(X,\Psi\II Y)\;\,\tilde\to\;\,\Hom_{\widetilde{\cD}}(\II Y,X)^\ast,$$
natural in $X$ and $Y$. Since $\widetilde\cD$ is generated as a triangulated category by the image of~$\II$, see Proposition~\ref{enough}(ii), this implies that~$\Psi$ is the Serre functor on~$\widetilde\cD$.

Conversely, equation~\eqref{HHHH} implies that, with $\widetilde\Phi$ the Serre functor of~$\widetilde\cD$, we have isomorphisms
$$\nu_{X,Y} :\;\Hom_{\widetilde\cD}(X,\CC\Phi Y)\;\,\tilde\to\;\,\Hom_{\widetilde{\cD}}(X,\widetilde\Phi\II Y),~$$
natural in $X$ and $Y$, from which we obtain a natural isomorphism $\CC\Phi\to \widetilde\Phi\II$. 
\end{proof}


\part{Applications to Lie superalgebras}

From now on we always work over~$\mk=\mC$.

\section{Gorenstein homological algebra for Lie superalgebras}\label{SecLSA}
In this section, we let $\fg=\fg_{\oa}\oplus \fg_{\ob}$ be a finite dimensional Lie superalgebra over~$\mC$, see \cite[Section~1.1]{book}. The subalgebra $\fg_{\oa}$ in degree $\oa\in\mZ_2$ is known as the underlying Lie algebra. 

\subsection{Supermodules over Lie superalgebras}
\subsubsection{}Let $U=U(\fg)$ and~$U_{\oa}=U(\fg_{\oa})$ denote the universal enveloping algebras. 
We denote by~$\fg$-sMod, or $U$-sMod, the category of all $\mZ_2$-graded $\fg$-modules, where the morphisms are given by the morphisms of~$\fg$-modules which respect the~$\mZ_2$-grading. The notation~$\Hom_{\fg}$ will be used for the space of these morphisms. The parity shift functor~$\Pi$ on~$\fg$-sMod is the exact functor which preserves every module, but reverses its $\mZ_2$-grading, and preserves every morphism. In particular,
$$\fg_{\oa}\mbox{-sMod} \;\cong\; \fg_{\oa}\mbox{-Mod}\;\oplus\; \Pi\fg_{\oa}\mbox{-Mod}.$$

\subsubsection{}For $M\in \fg$-sMod, we define $M^\ast=\Hom_{\mC}(M,\mC)$, with $\fg$-action given by 
$$(X\alpha)(v)=-(-1)^{|X||\alpha|}\alpha(Xv),$$ for~$X\in\fg$,~$v\in M$ and $\alpha\in M^\ast$. For $M,N\in \fg$-sMod, the module structure $M\otimes_{\mC} N$ is defined by 
$$X(v\otimes w)=Xv\otimes w+(-1)^{|X||v|}v\otimes Xw.$$ Clearly $M\otimes N\cong N\otimes M$. For a finite dimensional $V\in \fg$-sMod, we have
\begin{equation}\label{eqadjun}\Hom_{\fg}(M\otimes V,N)\,\cong\,\Hom_{\fg}(M,N\otimes V^\ast ).\end{equation}

We recall the following result of Bell and Farnsteiner.
\begin{lemma}\label{LemBell}\cite[Theorem~2.2]{Bell}
For any Lie superalgebra $\fk$ satisfying $\fg_{\oa}\subset \fk\subset \fg$, we have that~$U(\fg)$ is an $\alpha$-Frobenius extension of~$U(\fk)$. With $d:=\dim\fg_{\ob}-\dim\fk_{\ob}$,~$\alpha$ is determined by
$$\alpha(X)=\begin{cases}(-1)^dX&\mbox{for~$X\in\fk_{\ob}$},\\
X+\tr( \ad_{X}:\fg/\fk\to\fg/\fk)&\mbox{for~$X\in\fk_{\oa}$}.\\
\end{cases}$$
\end{lemma}

\subsubsection{} We will focus in particular on the case $\fk=\fg_{\oa}$. We consider the restriction functor
$$\Res:\;\fg\mbox{-sMod}\;\to \; \fg_{\oa}\mbox{-sMod}.$$
The left adjoint is $\Ind$ and the right adjoint is $\Coind$.
We will generally leave out the references to~$\fg$ and $\fg_{\oa}$ in these functors.
Let $K_{\fg}$ denote the one-dimensional $\fg_{\oa}$-module~$\Lambda^{\tp}\fg_{\ob}=\Lambda^{\dim \fg_{\ob}}\fg_{\ob}$. By using Lemma~\ref{LemBell} and keeping track of parity, we find an isomorphism of functors
\begin{equation}\label{IndCoind}\ind\;\cong\; \coind\circ (K_{\fg}\otimes -).\end{equation}
By the PBW theorem, we have
\begin{equation}\label{ResInd}\res\circ\ind\;\cong\; \Lambda\fg_{\ob}\otimes -.\end{equation}
Note that~$\mC=\Lambda^0\fg_{\ob}$ and $K_{\fg}=\Lambda^{\tp}\fg_{\ob}$ are direct summands of~$\Lambda\fg_{\ob}$. Hence, in this case, the monomorphic unit $\Id\hookrightarrow \RR\circ\II$ and epimorphic counit $\RR\circ\CC\tto \Id$ of Lemma~\ref{LemBasix} even split.

\subsection{Pairs of good module categories}

\subsubsection{}\label{DefGMC}
We are interested in abelian subcategories $\cB$ of~$\fg$-sMod with the following properties: 
\begin{enumerate}[(a)]
\item $\cB$ contains enough projective and injective objects;
\item all objects in~$\cB$ have finite length;
\item if $M\in\cB$, then $M\otimes V\,\in\cB$ for any finite dimensional $V\in \fg$-sMod.
\end{enumerate}
We will simply call categories satisfying these four properties {\bf good module categories}. By definition and Schur's lemma, good module categories are \lfp.

\begin{lemma}\label{LemGdV}
For a good module category~$\cC$, finite dimensional $V\in \fg\mbox{{\rm -sMod}}$ and $M\in\cC$, we have
\begin{enumerate}[(i)]
\item $\pd_{\cC} M\otimes V\;\le\; \pd_{\cC}M$;
\item $\Gd_{\cC} M\otimes V\;\le\; \Gd_{\cC}M$.
\end{enumerate}
\end{lemma}
\begin{proof}
For any finite dimensional $\fg$-module~$V$, the functor~$V\otimes -$ on a good module category~$\cC$ is exact and by \eqref{eqadjun} it maps projective modules to projective modules. These properties imply part~(i). Furthermore, by Definition~\ref{defGP} it follows that $-\otimes V$ maps Gorenstein projective modules to Gorenstein projective modules. Applying exactness again proves part~(ii).
\end{proof}

Similarly, it follows easily that
\begin{equation}\label{GEadj}\GE^k_{\cC}(M\otimes V,N)\;\cong\; \GE^k_{\cC}(M,N\otimes V^\ast),\end{equation}
for arbitrary~$M\in {}^g\cC$ and $N\in \cC$.

\begin{prop}\label{PropPair}
Consider a Lie superalgebra $\fg$ and a good module category~$\cC$ for $\fg_{\oa}$. Let~$\widetilde\cC$ denote the full subcategory of~$\fg${\rm -sMod} of modules $M$ satisfying $\res M\in \cC$.
\begin{enumerate}[(i)]
\item The category~$\widetilde\cC$ is a good module category. 
\item The adjoint pairs functors~$(\ind,\res)$ and~$(\res,\coind)$ restrict to functors between $\cC$ and~$\widetilde\cC$, which make $\widetilde\cC$ a Frobenius extension of~$\cC$.
\item The category~$\widetilde\cC$ is~$d$-Gorenstein if and only if $\cC$ is~$d$-Gorenstein.
\item For any~$M\in\cC$, we have
$$\pd_{\cC}M\;=\;\pd_{\widetilde\cC}\ind M\;=\;\pd_{\widetilde\cC}\coind M,$$
where the same holds for injective dimensions.
\end{enumerate}
\end{prop}
\begin{proof}
The fact that~$\ind$ and $\coind$ restrict to a functors~$\cC\to\widetilde\cC$ follows immediately from equations~\eqref{ResInd} and~\eqref{IndCoind} and property~(d) for~$\cC$. It then follows immediately from equation~\eqref{IndCoind} that~$\widetilde\cC$ is a Frobenius extension of~$\cC$, proving part~(ii).

By part (ii) and Proposition~\ref{PropFElfp}(i), $\widetilde\cC$ is \lfp.
Condition~\ref{DefGMC}(c) follows from the fact that~$\res$ commutes with tensor products, as~$U(\fg_{\oa})$ is a Hopf subalgebra of~$U(\fg)$. Hence part~(i) follows.

Part (iii) then follows from Propositions~\ref{PropPair}(ii) and \ref{PropFElfp}(ii). Part (iv) follows from Lemma~\ref{LemTriv} and the fact that~$\res\ind M$ contains $M$ as a direct summand by equation~\eqref{ResInd}. 
\end{proof}

A pair $(\widetilde\cC,\cC)$, where $\widetilde\cC$ is obtained from $\cC$ as in Proposition~\ref{PropPair} will be called a {\bf pair of good module categories}. 

\subsection{Lie superalgebras of type I}
We say that a Lie superalgebra $\fg$ is of {\bf type I} if it admits a three-term $\mZ$-grading compatible with the~$\mZ_2$-grading. Concretely, we have
$$\fg=\fg_{-1}\oplus\fg_{0}\oplus \fg_1,\qquad\mbox{with}\qquad\fg_{\oa}=\fg_0\;\mbox{ and }\;\, \fg_{\ob}=\fg_{-1}\oplus\fg_1.$$

\subsubsection{} For a Lie superalgebra of type I we introduce the exact parabolic induction functors 
$$\ind^{\pm}\,:\, \fg_0{\rm -sMod}\;\to\; \fg{\rm-sMod}.$$
The functor~$\ind^{\pm}$ corresponds to first interpreting $\fg_0$-modules as~$\fg_0\oplus\fg_{\pm1}$-modules with trivial $\fg_{\pm 1}$-action, followed by~$\ind^{\fg}_{\fg_0\oplus\fg_{\pm1}}$. Note that exactness of~$\ind^{\pm}$ follows from Lemma~\ref{LemBell}.
\begin{lemma}\label{LemI}
Let $(\widetilde\cC,\cC)$ be a pair of Gorenstein good module categories for a Lie superalgebra $\fg$ of type I. For any~$M\in\cC$, we have
$\Gd_{\widetilde\cC}\ind^{\pm} M\;=\;\Gd_{\cC} M.$
\end{lemma}
\begin{proof}
By Proposition~\ref{PropGoGo1}, we have
$$\Gd_{\widetilde\cC}(\ind^{+} M)\,=\, \Gd_{\cC}\res(\ind^+ M)\,=\, \Gd_{\cC}( \Lambda\fg_{-1}\otimes M).$$
As $M$ is a direct summand of~$\Lambda\fg_{-1}\otimes M$, Corollary~\ref{CorDS} implies that~$\Gd_{\cC} \Lambda\fg_{-1}\otimes M$ is at least $\Gd_{\cC} M$. The equality thus follows from Lemma~\ref{LemGdV}(ii).
\end{proof}

\begin{cor}\label{Corg1}
Keep the assumptions of Lemma~\ref{LemI}. If an arbitrary short exact sequence
$$0\to M\to N\to K\to 0$$
in~$\widetilde\cC$ is left $gp\widetilde\cC$-acyclical, then the induced sequences
$$0\to M^{\fg_1}\to N^{\fg_1}\to K^{\fg_1}\to 0\qquad\mbox{and}\qquad0\to M^{\fg_{-1}}\to N^{\fg_{-1}}\to K^{\fg_{-1}}\to 0$$
are exact and left $gp\cC$-acyclical in~$\cC$.
\end{cor}
\begin{proof}
Take an arbitrary~$P\in p\cC$. By Lemma~\ref{LemI}, we have~$\ind^+P\in gp\widetilde\cC$. As the short exact sequence is left $gp\widetilde\cC$-acyclical, we find by adjunction that
$$0\to \Hom_{\cC}(P,M^{\fg_1})\to \Hom_{\cC}(P,N^{\fg_1})\to \Hom_{\cC}(P,K^{\fg_1})\to 0$$
must be exact for all $P\in p\cC$. Hence~$0\to M^{\fg_1}\to N^{\fg_1}\to K^{\fg_1}\to 0$ is exact.
\end{proof}


\section{Super category~$\cO$}\label{SecO}

From now on we assume that~$\fg=\fg_{\oa}\oplus \fg_{\ob}$ is a {\bf classical} Lie superalgebra, see~\cite{book}. This means that~$\fg_{\oa}$ is a finite dimensional reductive Lie algebra and that the adjoint representation of~$\fg_{\oa}$ on~$\fg_{\ob}$ is finite dimensional and semisimple. We do {\bf not} require $\fg$ to be simple. 

\subsection{Parabolic category~$\cO$}

\subsubsection{}We consider a {\bf parabolic decomposition}
$$\fg\;=\; \fu^-\oplus\fl\oplus \fu^+,$$
as in~\cite[Section~2.4]{preprint}, with Levi subalgebra $\fl$ and parabolic subalgebra $\fp=\fl\oplus\fu^+$. Then~$\fu_{\oa}^-\oplus\fl_{\oa}\oplus \fu_{\oa}^+$ is a parabolic decomposition of $\fg_{\oa}$. If~$\fl_{\oa}$ is a Cartan subalgebra of $\fg_{\oa}$, we write
$$\fg\;=\; \fn^-\oplus\fh\oplus \fn^+,$$
for the {\bf triangular decomposition}, with Borel subalgebra $\fb=\fh\oplus\fn^+$.

We denote the Weyl group of~$\fg_{\oa}$ by~$W=W(\fg_{\oa}:\fh_{\oa})$. For any $w\in W$, its length is denoted by~$\ell(w)$. Let $w_0$ be the unique longest element of~$W$. In particular, we have $\ell(w_0)=\dim\fn_{\oa}^+$. For $\fp$ a parabolic subalgebra of~$\fg$, we have the corresponding longest element $w_0^{\fp}$ in the Weyl group of~$\fl_{\oa}$.

\subsubsection{}The category~$\scO(\fg,\fp)$ is the full subcategory of~$\fg$-sMod of modules which
\begin{itemize}
\item are finitely generated;
\item restrict to direct sums of simple finite dimensional $\fl_{\oa}$-modules;
\item are locally $U(\fu^+)$-finite.
\end{itemize}
With this definition, we have $\scO(\fg_{\oa},\fp_{\oa})=\cO(\fg_{\oa},\fp_{\oa})\oplus\Pi \cO(\fg_{\oa},\fp_{\oa})$, with $\cO(\fg_{\oa},\fp_{\oa})$ the usual parabolic subcategory of the BGG category of~\cite{BGG}.
Clearly,~$\scO(\fg,\fp)$ and $\scO(\fg_{\oa},\fp_{\oa})$ form a pair of good module categories, as by~\cite{BGG} the category~$\scO(\fg_{\oa},\fp_{\oa})$ is \lfp.

\subsection{Gorenstein homological algebra}
Fix a parabolic subalgebra $\fp$ of $\fg$.
\begin{thm}\label{ThmGO}
Set $d:=2\ell(w_0)-2\ell(w_0^{\fp})$.
\begin{enumerate}[(i)]
\item The \lfp category~$\scO(\fg,\fp)$ is $d$-Gorenstein.
\item The Gorenstein projective modules $M$ in~$\scO(\fg,\fp)$ are those for which $\res M$ is projective in~$\scO(\fg_{\oa},\fp_{\oa})$.
\item We have~$\Gd_{\scO(\fg,\fp)}M=\pd_{\scO(\fg_{\oa},\fp_{\oa})}\res M$ for any~$M\in \scO(\fg,\fp)$.
\end{enumerate}
\end{thm}
\begin{proof}
Since $\gd\scO(\fg_{\oa},\fp_{\oa})=d$, see \cite{SHPO4}, this follows from
Propositions~\ref{PropGdpd} and \ref{PropPair}(ii).
\end{proof}

\begin{prop}\label{sesKGMO}
For any~$M\in \scO(\fg,\fp)$, there exists $G\in \scO(\fg,\fp)$ with $\res G$ projective in~$\scO(\fg_{\oa},\fp_{\oa})$ and $K\in\scO(\fg,\fp) $ with finite projective dimension, which admit an exact sequence
$$0\to K\to G\to M\to 0.$$
\end{prop}
\begin{proof}
This is an immediate application of Proposition~\ref{ContFin}, using Theorem~\ref{ThmGO}.
\end{proof}


\begin{rem}
The direct classification of indecomposable Gorenstein projective modules in~$\scO(\fg,\fp)$ is a wild problem. Consider the special case, with $\fg=\mathfrak{gl}(m|n)$ and $\fp=\fg_0\oplus\fg_1$ such that~$\scO(\fg,\fp)$ is the category of all finite dimensional weight modules. As $\scO(\fg_0,\fg_0)$ is semisimple, all modules in~$\scO(\fg,\fp)$ are Gorenstein projective, but the category is generally of wild representation type.
\end{rem}

\subsection{Serre functors for classical Lie superalgebras}
\subsubsection{} In \cite[Theorem~5.9]{Serre}, Mazorchuk and Miemietz obtained an elegant expression for the Serre functor on the category~$\cD_{per}(\scO(\fg,\fp))$, for~$\fg$ 
in the list
\begin{equation}\label{listMM}
\mathfrak{gl}(m|n),\;\mathfrak{sl}(m|n),\;\mathfrak{psl}(n|n),\; \mathfrak{osp}(m|2n),\;\mathfrak{q}(n),\;\mathfrak{pq}(n),\;\mathfrak{sq}(n),\;\mathfrak{psq}(n).
\end{equation}
One can check directly that the condition in Proposition~\ref{PropSerreIC} is satisfied. Instead, we will derive an alternative expression for the Serre functor, which is also valid in slightly greater generality. 

\subsubsection{}\label{conKg}For the remainder of this section we consider an arbitrary classical Lie superalgebra $\fg$ for which the~$\fg_{\oa}$-module~$K_\fg=\Lambda^{\tp}\fg_{\ob}$ can be interpreted as a $\fg$-module. More precisely, the condition is that for the character $\gamma:\fg_{\oa}\to \mC$; $X\mapsto\tr(\ad_X:\fg_{\ob}\to\fg_{\ob})$, the subspace
$$\fg_{\ob}\oplus\ker\gamma\;\subset\; \fg$$ 
constitutes an ideal. This condition is satisfied for all algebras in~\eqref{listMM}, and allows us to introduce the functor~$\KK$ on $\scO(\fg,\fp)$, as well as on $\scO(\fg_{\oa},\fp_{\oa})$, as~$\KK=(K_{\fg}^\ast\otimes -)$, which intertwines the restriction functor and its adjoints. Moreover, we have $\KK\circ\ind\cong\coind$. In many cases, $\KK$ will be $\Id$ or $\Pi$.

\subsubsection{} For any simple reflection~$s\in W$, let $\TT_s$ denote the corresponding twisting functor on~$\scO(\fg,\fb)$ of \cite[Section~5]{Crelle}. By~\cite[Lemma~5.3]{Crelle}, these functors satisfy braid relations. Hence we can define $\TT_{w_0}$ by composing twisting functors corresponding to a reduced expression of~$w_0$. By~\cite[Lemma~5.4]{Crelle},~$\TT_{w_0}$ is right exact. In the following theorem we will restrict the cohomology functors~$\cL_i \TT_{w_0}$ on~$\scO(\fg,\fb)$ to the full subcategory~$\scO(\fg,\fp)$, for arbitrary parabolic subalgebras $\fp$.

\begin{thm}
Let $\fg$ be as in~\ref{conKg}. The restriction of~$\cL_{\ell(w_0^{\fp})}\TT_{w_0}$ to~$\scO(\fg,\fp)$, which we denote by~$\TT^{\fp}$, is right exact. 
Furthermore,~$\KK\circ\cL (\TT^{\fp})^2\cong\KK\circ\cL\TT^{\fp}\circ\cL\TT^{\fp}$ is a Serre functor on~$\cD_{per}(\scO(\fg,\fp))$.
In particular,~$\KK\circ\cL (\TT_{w_0})^2$ is a Serre functor on~$\cD_{per}(\scO(\fg,\fb))$.
\end{thm}
\begin{proof}
We denote the twisting functor on~$\scO(\fg_{\oa},\fb_{\oa})$ by~$\TT^{\oa}_{w_0}$. By \cite[Lemma~5.1]{Crelle}, we have
\begin{equation}\label{TResInd}\res\circ \TT_{w_0}\;\cong\;\TT^{\oa}_{w_0}\circ \res,\qquad\mbox{and}\;\qquad \ind\circ\TT^{\oa}_{w_0}\;\cong\; \TT_{w_0}\circ\ind.\end{equation}
As $\res$ and $\ind$ are exact functors mapping projective to projective modules, these properties immediately extend to the left derived functor and its cohomology functors. 

First we deal with the special case $\fp=\fb$. By \cite[Proposition~4.1(1)]{prinjective}, $\cL (T_{w_0}^{\oa})^2\cong\cL T_{w_0}^{\oa}\cL T_{w_0}^{\oa}$ is a Serre functor of $\cD_{per}(\scO(\fg_{\oa},\fb_{\oa}))$. The isomorphism between the two expressions of the Serre functor follows from \cite[Corollary~10.8.3]{Weibel} and the fact that~$T_{w_0}$ maps projective modules to~$T_{w_0}$-acyclical modules. The latter is well-known, see e.g.~\cite{prinjective, dualities}. By equation~\eqref{TResInd} we have 
$$\coind\circ\cL\TT^{\oa}_{w_0}\;\cong\;\KK\circ\cL \TT_{w_0}\circ\ind$$ as well as the analogue for the right adjoint of the twisting functors. We can thus apply Lemma~\ref{LemPhiPsi} to conclude that~$\KK\cL T_{w_0}$, and hence $\cL T_{w_0}$, is an auto-equivalence of $\cD_{per}(\scO(\fg,\fb))$ and then Proposition~\ref{PropSerreIC} to conclude that~$\KK\cL T_{w_0}\cL T_{w_0}$ is a Serre functor.

Now we consider the parabolic case.
The fact that the restriction of~$\cL_{\ell(w_0^{\fp})}\TT_{w_0}$ to~$\scO(\fg,\fp)$ is right exact follows from \eqref{TResInd}, the faithfulness of~$\res$ and the corresponding claim for~$\fg_{\oa}$ in~\cite[Lemma~8.4]{dualities}. Furthermore, just as above, Lemma~\ref{LemPhiPsi} and Proposition~\ref{PropSerreIC} imply that it suffices to prove that the left derived functor of the restriction of~$\cL_{\ell(w_0^{\fp})}\TT^{\oa}_{w_0}$ to~$\scO(\fg_{\oa},\fp_{\oa})$ yields a Serre functor of~$\cD_{per}(\scO(\fg_{\oa},\fp_{\oa}))\cong\cD^b(\scO(\fg_{\oa},\fp_{\oa}))$. 
That property follows mutatis mutandis the proof on p24-25 of \cite{prinjective}, using the case $\fp=\fb$ and \cite[Lemma~8.3]{dualities}.
\end{proof}






\section{Category~$\cO$ for the general linear superalgebra}\label{SecGL} Now we focus on the general linear superalgebra $\fg=\mathfrak{gl}(m|n)$, see~\cite[Section~2.2]{book}.


\subsection{Category~$\cO$ for~$\mathfrak{gl}(m|n)$} An overview of the theory of category~$\cO$ for~$\mathfrak{gl}(m|n)$ is given in~\cite{Brundan}.

\subsubsection{}The Lie superalgebra $\mathfrak{gl}(m|n)$ is of type I, with
$$\mathfrak{g}_0\cong\mathfrak{gl}(m)\oplus\mathfrak{gl}(n),\quad\mathfrak{g}_{1}\cong\mC^m\otimes(\mC^n)^\ast\;\mbox{ and}\quad \mathfrak{g}_{-1}\cong(\mC^m)^\ast\otimes \mC^n.$$ We choose the {\bf distinguished} Borel subalgebra $\fb=\fb_{\oa}\oplus\fg_1$, with $\fb_{\oa}$ given by all upper triangular matrices in~$\fg_0$. The corresponding Cartan subalgebra is~$\fh\cong\mC^m\oplus \mC^n$. We choose a corresponding basis~$\{\epsilon_1,\ldots,\epsilon_m,\delta_1,\ldots,\delta_n\}$ of~$\fh^\ast$. The even and odd positive roots are then given by
$$\Delta_{\oa}^+=\{\epsilon_i-\epsilon_j\,|\,i <j\}\cup \{\delta_i-\delta_j\,|\,i <j\}\quad\mbox{and}\quad \Delta_{\ob}^+=\{\epsilon_i-\delta_j\}.$$
We define a partial order on~$\fh^\ast$ by setting $\mu\le \lambda$ if and only if $\lambda-\mu$ is sum of elements in~$\Delta^+$. The Weyl group is~$W=W(\fg_{\oa}:\fh)\cong S_n\times S_m$. We let $z\in \fh$ denote the element which satisfies $\epsilon_j(z)=1$ and $\delta_k(z)=-1$. 

\subsubsection{} The {\bf anti-distinguished} Borel subalgebra is~$\bar{\fb}:=\fb_{\oa}\oplus\fg_{-1}$. As this corresponds to the distinguished Borel subalgebra for~$\mathfrak{gl}(n|m)\cong\mathfrak{gl}(m|n)$, all our results are valid for this choice as well. Moreover, by definition we have $\scO(\fg,\fb)=\scO(\fg,\bar{\fb})$. However, the Verma modules differ for both interpretations, as well as the labelling of the simple objects by highest weights.

\subsubsection{}Following~\cite[Section~2]{Brundan}, we can associate a parity~$p(\lambda)\in\mZ_2$ to each $\lambda\in\fh^\ast$ such that the weight space $M_\lambda$ in an indecomposable module~$M$ in $\scO$ is of parity~$p(\lambda)+p_M$, for some $p_M\in\mZ_2$ independent of $\lambda$. We define $\cO$ as the subcategory of $\scO$ of modules $M$ for which $p_M=\oa$. Consequently we have
$$\scO(\fg,\fb)\;=\;\cO(\fg,\fb)\oplus\Pi\cO(\fg,\fb),$$
see \cite[Lemma~2.2]{Brundan}. We simply write $\cO=\cO(\fg,\fb)$ and $\cO^{\oa}=\cO(\fg_{\oa},\fb_{\oa})$

\subsubsection{}For any~$\lambda\in\fh^\ast$ we have the Verma module~$M_0(\lambda):=U(\fg_{\oa})\otimes_{U(\fb_{\oa})} \Pi^{p(\lambda)}\mC_\lambda$, which has simple top~$L_0(\lambda)$. These $L_0(\lambda)$ are non-isomorphic for different $\lambda$ and exhaust all simple objects in~$\cO(\fg_{\oa},\fb_{\oa})$. The corresponding Verma module for~$\fg$ is given by
$$M(\lambda)\;=\;U(\fg)\otimes _{U(\fb)}\Pi^{p(\lambda)}\mC_\lambda\;\cong\; \ind^+ M_0(\lambda).$$
This has simple top~$L(\lambda)$ and these exhaust all simple objects in~$\cO(\fg,\fb)$. 
It follows easily that
\begin{equation}\label{invL}
L(\lambda)^{\fg_1}\;\cong\;L_0(\lambda),
\end{equation}
by considering the~$\mZ$-grading on~$L(\lambda)$ induced by~$z\in\fh$. We also introduce
$$K(\lambda)\;=\;\ind^+ L_0(\lambda),$$
which is a quotient of~$M(\lambda)$ and has simple top~$L(\lambda)$.
One shows that~$\res$ and $\ind$ map modules with Verma flags to modules with Verma flags. For any $\lambda\in \fh^\ast$, the indecomposable tilting module~$T(\lambda)$ is defined in \cite[Proposition~7(b)]{preprint}. We denote the injective envelope of $L(\lambda)$ by~$I(\lambda)$.

\subsubsection{} As $\fg$ has a Chevalley anti-automorphism, the category~$\cO$ admits a simple-preserving duality~$\dd$, see \cite[Section~13.7]{book}. Using this duality we can interpret Lemma~\ref{Extop} as
\begin{equation}
\label{eqExtd}
\GE^j_{\cO}(M,N)\;\cong\; \GE^j_{\cO}(\dd N,\dd M).
\end{equation}
By \cite{CMW}, it suffices to consider modules with weights in the set
$$\Lambda_0=\{\lambda\in \fh^\ast\,|\,\lambda=\sum_i\lambda_i\epsilon_i+\sum_j\lambda_{m+j}\delta_j,\;\,\mbox{with $\lambda_k\in\mZ$}\}.$$
We also denote by~$\Lambda^+_0$ (and $\Lambda^{++}_0$) the dominant (regular) weights in $\Lambda_0$ for the dot action, see \cite[Section~15.3]{book}. An element of $\Lambda_0$ is called typical if there is no $\gamma\in\Delta^+_{\ob}$ such that~$\langle \lambda+\rho,\gamma\rangle=0.$

\subsubsection{Translation functors}\label{TransFun} Let $U\cong \mC^{m|n}$ be the natural representation of~$\mathfrak{gl}(m|n)$, we have functors~$\FF=-\otimes U$ and $\EE=-\otimes U^\ast$ on~$\cO$, which decompose, following~\cite[Section~2.8]{Kujawa}, as~$\EE=\oplus_{i\in\mZ}\EE_i$ and $\FF=\oplus_{i\in\mZ}\FF_i$. By definition and equation~\eqref{GEadj}, we have
\begin{equation}\label{GEFE}\GE^j_{\cO}(\EE_iM,N)\cong \GE^j_{\cO}(M,\FF_iN)\qquad\mbox{and}\qquad \GE^j_{\cO}(\FF_iM,N)\cong\GE^j_{\cO}(M,\EE_iN).\end{equation}

\subsection{G-dimensions}
Projective dimensions do not yield information on atypical simple or Verma modules, see \cite[Theorem~6.1]{CS}. The follow theorem shows that G-dimensions can resolve that. 
\begin{thm}\label{ThmGdO}${}$
\begin{enumerate}[(i)]
\item $\Gd_{\cO}K(\lambda)=\pd_{\cO^{\oa}} L_0(\lambda)$, so for~$\mu\in \Lambda_0^{++}$, we have $\Gd_{\cO}K(w\cdot\mu)=2\ell(w_0)-\ell(w)$;
\item $\Gd_{\cO} \Delta(\lambda)=\pd_{\cO^{\oa}} \Delta_0(\lambda)$, so for~$\mu\in \Lambda_0^{++}$, we have $\Gd_{\cO}\Delta(w\cdot\mu)=\ell(w)$;
\item $\Gd_{\cO}L(\lambda)\ge \pd_{\cO^{\oa}} L_0(\lambda)$;
\item $\Gd_{\cO}I(\lambda)=\pd_{\cO^{\oa}}I_0(\lambda)$.
\item $\Gd_{\cO}T(\lambda)=\pd_{\cO^{\oa}}T_0(\lambda)$.
\end{enumerate}
\end{thm}
\begin{proof}
The G-dimensions of~$K(\lambda)$ and $\Delta(\lambda)$ follow immediately from Lemma~\ref{LemI} and \cite[Propositions~3 and~6]{SHPO1}. 
It is clear that~$L_0(\lambda)$ is a direct summand of~$\res L(\lambda)$. Property (iii) thus follows from Proposition~\ref{PropGdpd}. Part (iv) follows from \cite[Theorem~6.1(iii)]{CS} and Corollary~\ref{CorFinDim}(ii). For part~(v), we can observe that~$T_0(\lambda)$ is a direct summand of~$\res T(\lambda)$, whereas~$T(\lambda)$ is a direct summand of~$\ind (T_0(\lambda)\otimes \Lambda^{\tp}\fg_{-1})$.
\end{proof}

\begin{rem}
The values $\pd_{\cO^{\oa}} L_0(\lambda)$ and $\pd_{\cO^{\oa}} \Delta_0(\lambda)$ are presently only explicitly known for special cases, but can in principle be computed from Kazhdan-Lusztig combinatorics. The values~$\pd_{\cO^{\oa}} I_0(\lambda)$ and $\pd_{\cO^{\oa}} T_0(\lambda)$ are known in terms of Lusztig's a-function. An overview is given in~\cite{SHPO4}.
\end{rem}

\begin{cor}
For any~$\lambda\in\Lambda$, we have
$\Gd_{\cO}L(\lambda)=2\ell(w_0)$ if and only if $\lambda\in\Lambda_0^{++}.$
\end{cor}
\begin{proof}
If~$\lambda\in\Lambda_0^{++}$, then \cite[Proposition~3]{SHPO1} implies that~$\pd_{\cO^{\oa}}L_0(\lambda)=2\ell(w_0)$, which is the maximal G-dimension. The equality~$\Gd_{\cO}L(\lambda)=2\ell(w_0)$ thus follows from Theorem~\ref{ThmGdO}(iii).

Now assume that~$\lambda\not\in\Lambda_0^{++}$ and consider $Q$ defined by the short exact sequence
\begin{equation}\label{sesQdK}0\to L(\lambda)\to \dd K(\lambda)\to Q\to 0.\end{equation}
One finds
$$\pd_{\cO^{\oa}}\res \dd K(\lambda)\;=\;\pd_{\cO^{\oa}}\Lambda\fg_{-1}\otimes L_0(\lambda)\;=\;\pd_{\cO^{\oa}} L_0(\lambda).$$
By \cite[Section~1]{SHPO4} we have $\pd_{\cO^{\oa}} L_0(\lambda)<2\ell(w_0)$.
Now assume that~$\pd_{\cO^{\oa}}\res L(\lambda)=2\ell(w_0)$. The short exact sequence~\eqref{sesQdK} then implies $\pd_{\cO^{\oa}}\res Q=2\ell(w_0)+1$, strictly bigger than the global dimension of~$\cO^{\oa}$, a contradiction.
\end{proof}

\begin{ex} Despite the above corollary, there are examples where the inequality in Theorem~\ref{ThmGdO}(iii) will be strict.
For $\fg=\mathfrak{gl}(2|1)$ one computes
$$\res L(-\epsilon_1)=L_0(-\epsilon_1)\oplus L_0(-2\epsilon_1+\delta).$$
Hence we find
$$\Gd_{\cO} L(-\epsilon_1)=\pd_{\cO^{\oa}}L_0(-2\epsilon_1+\delta)=1\;>\; 0=\pd_{\cO^{\oa}}L_0(-\epsilon_1).$$
Nevertheless, it is clear that for weights which are `generic', see e.g.~\cite[Section~7]{Crelle}, the inequality in Theorem~\ref{ThmGdO}(iii) will be an equality.
\end{ex}

\subsection{First Gorenstein extensions of simple modules}

\begin{prop}\label{PropOutside} Take $\lambda\in\Lambda_0$.
Let $N\in \cO$ be such that~$[N:L(w\cdot\lambda)]=0$ for all $w\in W\backslash\{e\}$ and furthermore $[N:L(\nu)]=0$ for~$\nu<\lambda$. Then we have
$\GE^1_{\cO}(L(\lambda),M)=0.$
\end{prop}
\begin{proof}
Consider a left $gp\cO$-acyclical short exact sequence
$$\xymatrix{0\ar[r]& N\ar[r]^{\alpha}& M\ar[r]^\beta& L(\lambda)\ar[r]& 0.}$$
By Corollary~\ref{Corg1} and equation~\eqref{invL}, taking $\fg_1$-invariants yields a short exact sequence
$$0\to N^{\fg_1}\to M^{\fg_1}\to L_0(\lambda)\to 0.$$
By assumption and~\eqref{invL} all simple factors in $N^{\fg_1}$ have highest weight either $\lambda$ or weights in different Weyl group orbits. By linkage in $\cO^{\oa}$, see e.g.~\cite{BGG}, this sequence must split, so we can take
$$\gamma\in \Hom_{\fg}(K(\lambda),M)\cong \Hom_{\fg_{\oa}}(L_0(\lambda),M^{\fg_1}),$$
such that~$\beta\circ\gamma\not=0$.
By assumption, no factor in the radical of~$K(\lambda)$ appears in $M$. Hence,~$\gamma$ factors through a morphism $\overline{\gamma}:L(\lambda)\to M$ which splits the short exact sequence.

So we find that any left $gp\cC$-acyclical extension of~$L(\lambda)$ and $M$ vanishes and the conclusion follows from Proposition~\ref{LemExt1}.
\end{proof}

\begin{prop}\label{Extorbit}
Consider $\lambda,\mu\in\fh^\ast$ which belong to different Weyl group orbits, then
$$\GE_{\cO}^{1}(L(\lambda),L(\mu))=0.$$ 
\end{prop}

\begin{proof}
By equation~\eqref{eqExtd}, we can assume that~$\mu\not\le\lambda$ and take $M=L(\mu)$ in Proposition~\ref{PropOutside}.
\end{proof}

There are also natural restrictions on the first Gorenstein extensions between modules with highest weight inside the same Weyl group orbit.
\begin{prop}\label{PropExtOrbit}
Consider $\lambda,\mu\in\fh^\ast$ which belong to the same Weyl group orbit, then
$$\dim\GE_{\cO}^{1}(L(\lambda),L(\mu))\;\;\le\;\;\dim\Ext_{\cO}^{1}(L(\lambda),L(\mu))=\dim \Ext^1_{\cO^{\oa}}(L_0(\lambda),L_0(\mu)). $$ 
When the weights are typical, the inequality is actually an equality.
\end{prop}
\begin{proof}
The right equality is
is \cite[Lemma~3.9]{CS}. The inequality is a special case of Proposition~\ref{PropTateGExt}.
It is well known that typical modules have finite projective dimension, see e.g. \cite[Theorem~6.1(i)]{CS}. The claim for typical modules thus follows form Lemma~\ref{LemExtProp}(ii).
\end{proof}
We show that the weak inequality in Proposition~\ref{PropExtOrbit} can not be replaced by an equality.
\begin{lemma}\label{ExtTriv}
If~$mn\not=0$, we have $\GE^1_{\cO}(\mC,L)=0$ for any simple $L\in\cO$, with $\mC\cong L(0)$ the trivial $\fg$-module.
\end{lemma}
\begin{proof}
By Proposition~\ref{Extorbit} it suffices to prove that~$\mC$ has no first Gorenstein extensions with the modules $L(\lambda)$, where $\lambda$ is in the~$0$-orbit.
We can calculate the highest weight of $L(\lambda)$ with respect to the anti-distinguished Borel subalgebra, by using odd reflections, see e.g. \cite[Section~3.4]{book}
and \cite[Lemma~2.3]{Crelle}. The module~$\mC$ has of course highest weight $0$ in any root system. It follows quickly that the only highest weight module which has highest weight in the~$0$-orbit for both systems of positive roots is~$0$. We can thus apply Proposition~\ref{Extorbit} to exclude the corresponding possible extensions. Self-extensions of~$\mC$ are excluded by Proposition~\ref{PropExtOrbit}.
\end{proof}

\begin{prop}\label{Propm1}
For $\fg=\mathfrak{gl}(m|1)$, a finite dimensional atypical simple $L\in\cO$ and an arbitrary simple $L'\in\cO$, we have
$$\GE^1_{\cO}(L,L')=0.$$
\end{prop}
\begin{proof}
Up to an unimportant shift in the action of the centre, any atypical simple finite dimensional module is of the form
$$L\,=\, L(k_1\epsilon_1+k_2\epsilon_2+\cdots+k_m\epsilon_m+i\delta),$$
for some $i\in[0,m-1]$, with $k_i=0$ and $k_1\ge k_2\ge\cdots\ge k_m$. It is a straightforward exercise, using~\cite[Theorem~5.2]{Kujawa}, that for this $L$ there exists a composition~$\AAf$ of indecomposable translation functors as in \ref{TransFun}, such that 
$L\cong\AAf(\mC)$. Furthermore, for the adjoint translation functor~$\BB$, we find that~$\BB( L')$ is simple, for all simple highest weight modules $L'$ with highest weight in the orbit of  
$k_1\epsilon_1+k_2\epsilon_2+\cdots+k_m\epsilon_m+i\delta$. The result thus follows from equation~\eqref{GEFE}, Proposition~\ref{PropExtOrbit} and Lemma~\ref{ExtTriv}.
\end{proof}
\begin{rem}
The computation with translation functors is quite tedious. However, the conceptual reason that simple modules are mapped to simple modules is rather straightforward. It follows from the observation that at each stage we translate an atypical module with regular highest weight to a module with the same properties.
\end{rem}

\begin{cor}
For $\fg=\mathfrak{gl}(2|1)$ and simple modules $L,L'$, with at least one atypical, we have
$$\GE^1_{\cO}(L,L')=0.$$
\end{cor}
\begin{proof}
There are three possibilities for~$L=L(\lambda)$ atypical. Either $\lambda$ is dominant regular, singular or anti-dominant regular. Using Proposition~\ref{Extorbit}, the second case leads to vanishing of extension since there are no self-extensions. By Proposition~\ref{Propm1}, the first case vanishes. Finally, the third case also vanishes by equation~\eqref{eqExtd}, because now Proposition~\ref{Extorbit} requires $L'$ to be finite dimensional.
\end{proof}

\begin{flushleft}
	K. Coulembier\qquad \url{kevin.coulembier@sydney.edu.au}
	
	School of Mathematics and Statistics, University of Sydney, NSW 2006, Australia

\end{flushleft}

\end{document}